\documentclass[10pt]{amsart}
\usepackage{amsmath,amsfonts,amssymb,amscd,amsthm,amsbsy,amsxtra,bbm,bm, epsf,calc,comment,dsfont}
\usepackage{color}
\usepackage{datetime}
\usepackage{latexsym}
\usepackage[english]{babel}
\usepackage{enumerate}
\usepackage{graphicx}
\usepackage{hyperref}
\usepackage{epsfig}
\definecolor{pku}{RGB}{139,0,18}
\definecolor{purple}{RGB}{154,0,154}
\definecolor{green}{RGB}{0,154,0}
\definecolor{orange}{RGB}{255,127,0}

\let\f=\frac

\def\e{\epsilon}
\def\R{\mathbb{R}}

\def\II{{\rm I\kern-0.5exI}}
\def\III{{\rm I\kern-0.5exI\kern-0.5exI}}

\newcommand{\norm}[1]{\lVert #1 \rVert}

\newcommand{\RR}{\mathbb{R}}

\newcommand{\BR}{\mathbb{R}}

\DeclareMathOperator*{\argmin}{argmin}

\DeclareMathOperator*{\argmax}{argmax}

\DeclareSymbolFont{bbold}{U}{bbold}{m}{n}
\DeclareSymbolFontAlphabet{\mathbbold}{bbold}

\newcommand{\beq}{\begin{equation}}
\newcommand{\eeq}{\end{equation}}
\newcommand{\beqo}{\begin{equation*}}
\newcommand{\eeqo}{\end{equation*}}

\renewcommand{\th}{\theta}

\newcommand{\id}{id}

\newcommand{\vp}{\varphi}
\newcommand{\tr}{tr}
\DeclareMathOperator{\sgn}{sgn}

\numberwithin{equation}{section}
\newtheorem{theorem}{Theorem}[section]
\newtheorem{lemma}[theorem]{Lemma}
\newtheorem{prop}[theorem]{Proposition}
\newtheorem{corollary}[theorem]{Corollary}

\theoremstyle{remark}
\newtheorem{remark}[theorem]{Remark}

\theoremstyle{definition}
\newtheorem{definition}[theorem]{Definition}

\title{Tumor Growth with Nutrients: Regularity and Stability}

\author[M.\;Jacobs]{Matt Jacobs}
\address{Department of Mathematics, Purdue University, West Lafayette, IN}
\email{jacob225@purdue.edu}

\author[I.\;Kim]{Inwon Kim}
\address{Department of Mathematics, UCLA, Los angeles, CA}
\email{ikim@math.ucla.edu}

\author[J.\;Tong]{Jiajun Tong}
\address{Beijing International Center for Mathematical Research, Peking University, Beijing, China}
\email{tongj@bicmr.pku.edu.cn}

\thanks{I.\;Kim is partially supported by NSF Grant DMS-1900804.
J.\;Tong is partially supported by the Peking University Start-up Fund.}

\begin{document}

\begin{abstract}
In this paper we study a tumor growth model with nutrients. The model presents dynamic patch solutions due to the contact inhibition among the tumor cells. We show that when the nutrients do not diffuse and the cells do not die, the tumor density exhibits regularizing dynamics. In particular, we provide contraction estimates, exponential rate of asymptotic convergence, and boundary regularity of the tumor patch. These results are in sharp contrast to the models either with nutrient diffusion or with death rate in tumor cells.
\end{abstract}
\maketitle
\section{Introduction}

A model system that appears in literature describing tumor growth with nutrients is
\beq\label{eqn: P}
\left\{\begin{array}{l}
\rho_t - \nabla\cdot(\rho \nabla p) = (n- b) \rho,\quad \rho\leq 1,\quad p\in P_\infty (\rho),\\ \\
n_t - D\Delta n = -\rho n, \quad n\to c>0 \mbox{ as } |x| \to\infty
\end{array}\right.\tag{P}
\eeq
set in $Q:=\R^d \times [0,\infty)$, with $b$, $D$ and $c$ being non-negative constants (see e.g.\;\cite{MRS14, PQV, PTV14,david2021free}).
It is equipped with the initial conditions
\beq
\rho(x,0) = \rho_0(x),\quad n(x,0) = n_0(x).
\label{eqn: IC}
\eeq
Here $\rho$ and $n$ respectively denote the density of tumor cells and the nutrients.
The cells grow by consuming the nutrients which are supplied from the external environment, while they die at a constant rate $b$ in the meantime.
The pressure $p\geq 0$ can be understood as the Lagrange multiplier for the constraint $\rho \leq 1$ that represents the contact inhibition in cells.
In \eqref{eqn: P}, $P_\infty(\rho)$ denotes the Hele-Shaw graph
\[
P_\infty(\rho) =
\begin{cases}
0, & \mbox{if } \rho\in [0,1), \\
[0,+\infty), & \mbox{if } \rho = 1.
\end{cases}
\]
It is well-known that the resulting solution features time-evolving patches of congested cell region where $\rho$ equals $1$. We call this set a {\it tumor patch} for later reference.

\medskip

 This model, while relatively simple, presents a complex phenomena. One can view the system as a singular limit of a reaction-diffusion system, where one first takes $p=\rho^{m-1}$ and then sends $m$ to $+\infty$; see \cite{PQV,david2021free}.  With finite $m$, the reaction-diffusion system has been actively studied  in the literature: see \cite{KMMUS,kitsu97, M04}.

 \medskip

  The well-posedness of \eqref{eqn: P} is not hard to achieve: see Section~\ref{sec: summary of results} for more discussions. On the other hand, qualitative behavior of the solutions of \eqref{eqn: P} is much less understood. In particular, the growth of the tumor patch appears to generate fingering phenomena, as observed by numerical experiments \cite{kitsu97, MRS14, PTV14} even when the patch is almost radial and when $n$ is initially a constant. Such ``dendric growth" is well-known to persist in models of bacterial growth that aggressively consumes nutrients \cite{Ben94}. When $D>0$, the formation of dendritic patterns is conjectured to occur because there are more nutrients available near the tips of dendritic fingers compared to valleys \cite{MRS14} (in the valleys there are more surrounding bacteria to consume the nutrients).  As a result, dendritic tips grow faster than the valleys, leading to the amplification of instabilities.  When $D=0$ and $b>0$, numerical experiments still observe the fingering phenomena, possibly due to the  movement of tumor cells toward the necrotic core, where cells decay from the maximal density (c.f.~Figure \ref{fig:death_dendrites}).  The scale of the aforementioned instabilities in terms of $b$ and $D$ remains to be understood. While we do not pursue regularity analysis in such general cases, a variational scheme is introduced to approximate \eqref{eqn: P} and yield well-posedness in general settings: see Theorem~\ref{thm:wellposedness}.  Although there are many other possible ways to approximate the equation and obtain the well-posedness, the variational scheme we introduce is particularly numerically efficient and preserves many desirable properties of the true system (c.f.~the discussion in Section \ref{ssec:wp}).

  \begin{figure}
   \centering
    \includegraphics[width=.25\textwidth]{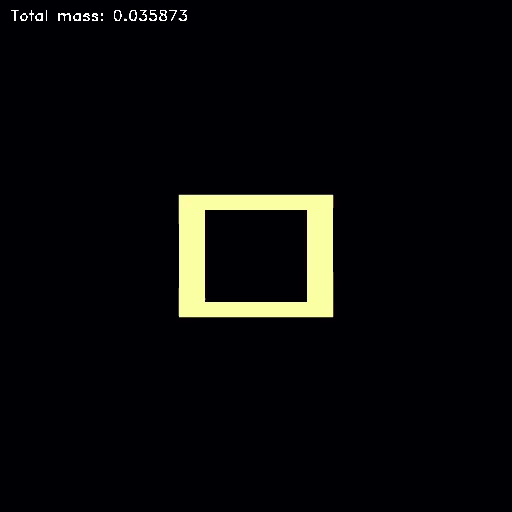}
      \includegraphics[width=.25\textwidth]{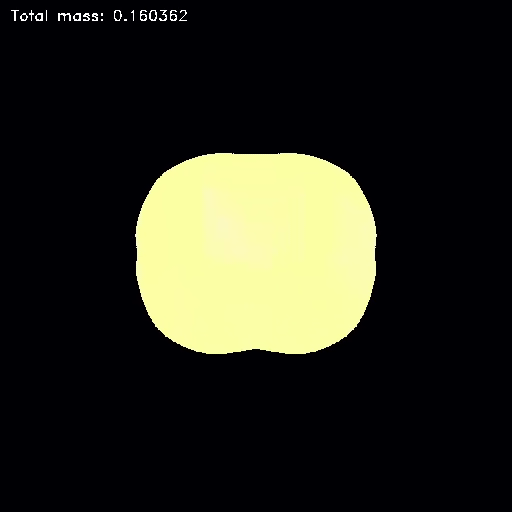}
      \includegraphics[width=.25\textwidth]{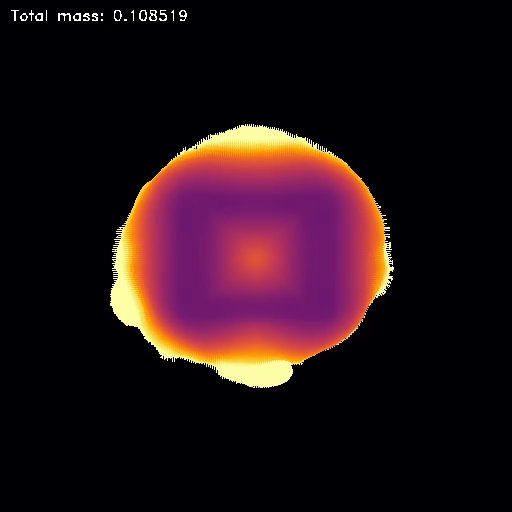}
      \includegraphics[width=.25\textwidth]{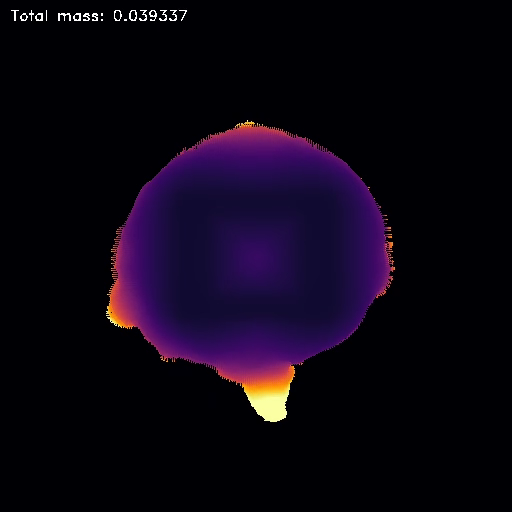}
       \includegraphics[width=.25\textwidth]{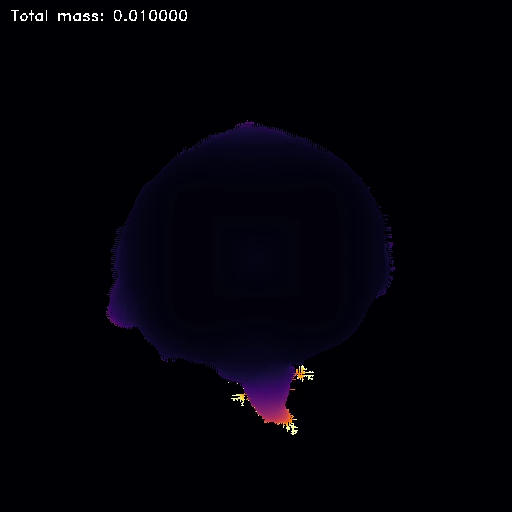}
        \includegraphics[width=.25\textwidth]{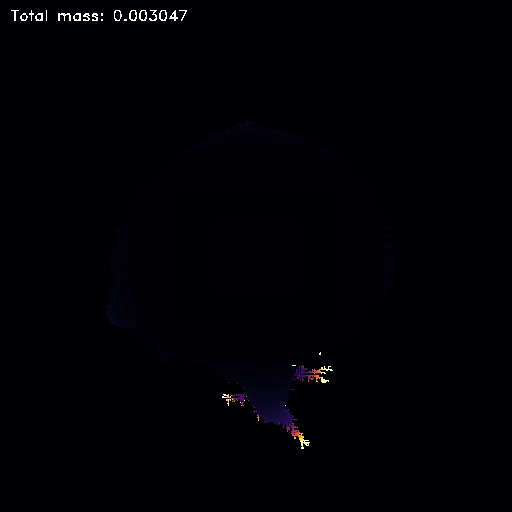}
    \caption{\footnotesize Numerical simulations from \cite{jacobs_lee} of the tumor growth system with $D=0$ and $b=0.4$ and $n_0=2$. The first image represents the initial patch density.  Brighter colored pixels indicate larger density values. The evolution shows dendritic growth at the boundaries once cells begin to die in the interior.}
    \label{fig:death_dendrites}
  \end{figure}


 \medskip

 The main goal of this paper is to study the singular case $b=D=0$, with particular focus on the dynamics and regularity of the tumor patch evolution.  This case is particularly interesting, as numerical experiments in \cite{MRS14} suggest that, when $b$ is fixed to be zero, the dendritic behavior becomes more and more branched and irregular as the diffusion parameter $D$ becomes smaller and smaller. Surprisingly, we find that once the diffusion parameter is set to zero, there is no dendritic growth whatsoever and the evolution is regularizing.
 Roughly speaking, we will show that there is no rough growth of the tumor patch other than those caused by topological changes (Theorem \ref{thm:reg}). Moreover, when $n$ is initially a constant, we can further show that the dynamics of the tumor patch can be understood in terms of a single-variable nutrient-free system (Theorem \ref{thm:master_dynamics}).   We hope our findings serve as the first step to understanding the complex behavior of the system \eqref{eqn: P} with general values of $b$ and $D$. In particular, reconciling our results with the numerical experiments in \cite{MRS14} would be a very interesting future direction of study.

 \medskip

Our results are based on the following rather unexpected {\it comparison principle} when $b= D = 0$  (Proposition~\ref{prop:comparison}), for two different solutions $(\rho^i,p^i,n^i)$ $(i=0,1)$ of \eqref{eqn: P}:
\[
\mbox{ Let }\eta^i:= n^i_0 - n^i.\mbox{ If }\rho^1 \leq \rho^0\mbox{ and }n^1\leq n^0\mbox{ at }t=0\mbox{,  then }\rho^1 \leq \rho^0\mbox{ and }\eta^1 \leq \eta^0\mbox{ for all }t\geq 0.
\]
This comparison property is somewhat unintuitive, as the smaller density should have more nutrients available for growth, raising the posibility that the ordering could be violated as the system evolves. As such, observe that this is not a standard comparison principle. While the initial ordering is with the nutrient variable,  the ordering at later times are associated with the $\eta$-variable.
To see why this should hold, we introduce the variable $w(x,t):= \int_0^t p(x,s)\, ds$.  Integrating the $\rho$-equation of \eqref{eqn: P} in time and noticing that $n\rho = -n_t=\eta_t$ when $D = 0$, we obtain
\beq
\rho - \Delta w  = \rho_0 + \int_0^t (n\rho)\, ds = \rho_0 +\eta,
\label{eqn:elliptic equation}
\eeq
and
\begin{equation}\label{eqn:eta}
\eta_t = (n_0-\eta)\rho.
\end{equation}
The time integrated system reveals that the total growth of the density at time $t$ only depends on $\eta$, the total amount of nutrients consumed by the time $t$, rather than the amount of available nutrients at time $t$. Thus, by working with this version of the problem, the possibility of a comparison property becomes much more evident.
Indeed, a large part of our analysis, including the proof of the above comparison principle, will be based on this time-integrated system \eqref{eqn:elliptic equation}-\eqref{eqn:eta}.

\medskip

Heuristically speaking, since $\rho$ only grows and the set $\{\rho=1\}$ expands in time, the pressure variable is positive in the growing parts of $\{\rho=1\}$.
More precisely, the set $\{\rho=1\}$ can be decomposed as the union $\{w >0\}\cup \{\rho_0=1, n_0=0\}$ (c.f.~Lemma \ref{lem:rho_1_characterization}).
On the other hand, $\rho = \rho_0 + \eta$ in $\{w=0\}$. Thus, $(w,\eta)$ satisfies
\beqo
\left\{\begin{array}{l}
(1-\eta -\rho_0)\chi_{\{w(\cdot,t)>0\}} -\Delta w(\cdot,t) = 0 \quad \hbox{in } \mathbb{R}^d,\\ \\
\eta_t = (n_0-\eta)\rho.
\end{array}\right. 
\eeqo
It is then not difficult to see that the aforementioned comparison principle holds for this system.
Since the above system is in the form of an obstacle problem, let us discuss the free boundary regularity of the set $\{w>0\}$. 
The standard theory for the obstacle problem yields that, as long as $\eta$ is less than $1$ near the boundary of $\{w>0\}$ and is $C^{\alpha}$, the boundary of $\{w>0\}$ has $C^{1,\alpha}$-regularity, away from cusp-type  singular points \cite{Blank,Caff98}. Thus, when we study the boundary regularity of the tumor patches, an important step in our analysis is to ensure the regularity of $n$, the non-degeneracy of $w$, and the nonexistence of cusp points on the patch boundary.

\medskip

We will see in Theorem~\ref{thm: convergence when n_0 is less than 1} and Corollary~\ref{cor:1} that, if $n_0(x)\geq 1$, the pressure dominates the evolution and there is a generic (local-in-time) regularity of the patch boundary after some finite time, whereas for $n_0(x)<1$, the pressure diminishes exponentially fast as the nutrient vanishes, and the tumor patch approaches an asymptotic profile.
The finite-time regularization result for $n_0\geq 1$ features similarity to those of porous medium equations \cite{CVW87} and the Hele-Shaw flow \cite{K06}.
In the case of $n_0<1$, the large-time regularity of the tumor patch remains open in general, and it may be possible that fractal structures persist as the set approaches its asymptotic profile. Nonetheless, if the asymptotic tumor patch is sufficiently far away from the convex hull of its initial position, we can show that the patch evolution turns smooth within finite time.

\medskip

Our last main result focuses on the problem when the initial nutrient is constant, namely when $n_0(x) \equiv c$. In this case, very surprisingly, one can characterize the behavior of the tumor patch solution in \eqref{eqn: P} through what we call \emph{master dynamics}.
There are two of them, stated in Proposition \ref{prop: master dynamics} and Proposition \ref{prop: master dynamics 2} respectively.
Let us take the first one as an example.
Consider the following problem that concerns the density evolution only:
\beq\label{eqn: HS}
\partial_t \rho_* -\nabla\cdot (\rho_*\nabla p_* ) = \rho_0,\quad \rho_*\leq 1,\quad p_*\in P_\infty(\rho_*),\quad \rho_*|_{t = 0} = \rho_0.  \tag{HS}
\eeq
$\rho_0$ is the initial data in \eqref{eqn: IC} which is assumed to be a patch here.
\eqref{eqn: HS} is reminiscent of the classic Hele-Shaw flow, whose regularity and long-time convergence property is relatively well-understood in various  settings \cite{EJ, K06, CJK07}.
We will show in Proposition \ref{prop: master dynamics} that, once $n_0 = c>0$ is given, the corresponding $\rho$-evolution in \eqref{eqn: P} is simply a re-scaled (in time) version of the $\rho_*$-evolution in \eqref{eqn: HS}, and the re-scaling depends on $n_0$ explicitly.
The $n$-evolution in \eqref{eqn: P} can be readily represented as well.
In other words, once we understand the density evolution in \eqref{eqn: HS}, which is parameter-free and much simpler than \eqref{eqn: P}, we can fully characterize all the $(\rho,n)$-evolutions in \eqref{eqn: P} corresponding to all different values of $n_0>0$.
This is why we call $\{\rho_*(\cdot,t)\}_{t\geq 0}$ the master dynamics.
The second master dynamics is proposed in a similar spirit; see Proposition \ref{prop: master dynamics 2}.
Given that the value of $n_0$ has a non-trivial impact on the patch solution dynamics in \eqref{eqn: P}, and that $n$ does not stay as a constant over time, this connection is far from being apparent.
We should emphasize that neither of the master dynamics can be obtained by a mere change of space and time variables in \eqref{eqn: P}, and even the relation between the two master dynamics is also highly non-trivial.
Their proof relies on a very special property of the system \eqref{eqn: P} when $b = D = 0$, $n_0(x)$ is a constant, and $\rho$ is a patch solution (see Lemma \ref{lem: general test function}): for any harmonic function, its average on the set $\{\rho =1\}$ is time-invariant.

\medskip

The rest of the paper is organized as follows. In Section \ref{sec: summary of results}, we state our main results and discuss their implications. In Section \ref{sec: discrete in time scheme}, we introduce a discrete-in-time variational scheme to approximate \eqref{eqn: P} in the general setting of $b,D\geq 0$ and prove well-posedness.  Sections \ref{sec: D is 0}-\ref{sec: master dynamics} are devoted to the case $b=D=0$. In Section \ref{sec: D is 0}, we first study the elliptic equation \eqref{eqn:elliptic equation} (see a more precise formulation in \eqref{eq:elliptic}), and then prove general properties of the time-integrated system when the initial nutrient is bounded. 
Section \ref{sec:regularity} is focused on free boundary regularity of the tumor patch solutions under suitable geometric assumptions.
Finally, in Section \ref{sec: master dynamics}, we prove the master dynamics when $n_0$ is constant.
Then we apply that to characterize long-time behavior and uniform boundary regularity of the patch solution $\rho$.

\section{Summary of Main Results}
\label{sec: summary of results}
\subsection{Well-posedness in the general setting}\label{ssec:wp}

We first define weak solutions of \eqref{eqn: P} as follows.
\begin{definition}\label{def: weak solution}
Let $n_0\in L^{\infty}(\R^d)\cap BV(\BR^d)$ and $\rho_0\in L^1(\BR^d)\cap BV(\RR^d)$ such that $\rho_0\in [0,1]$ almost everywhere.
Fix $T>0$ and denote $Q_T:=\BR^d\times [0,T]$.
Non-negative functions $\rho \in C([0,T]; L^1(\RR^d))\cap L^\infty([0,T];BV(\BR^d))$, $p\in L^2([0,T]; H^1(\RR^d))$, and $n\in L^{\infty}(Q_T)\cap L^\infty([0,T];BV(\BR^d))$ on $Q_T$ are said to form a weak solution of \eqref{eqn: P} in $Q_T$, if they satisfy:
\begin{enumerate}[(i)]
\item $\rho\in [0, 1]$ and $p(1-\rho)=0$ in $Q_T$;
\item For any $\psi\in H^1(Q_T)$ that vanishes at $t=T$, we have
\[
\int_0^T\int_{\RR^d}\nabla \psi\cdot \nabla p- \rho\partial_t \psi \,dx\,dt =\int_{\RR^d} \psi(x,0)\rho_0\,dx
+\int_0^T\int_{\RR^d} \psi (n-b)\rho\,dx\,dt;
\]
\item In addition,
\[
\partial_t n-D\Delta n=-\rho n \hbox { in }  \mathcal{D}'(Q_T), \quad n(\cdot,0) = n_0.
\]
\end{enumerate}
\end{definition}
Here and in what follows, by saying a function $f\in BV(\BR^d)$, we mean that the total variation of $f$ is finite, and yet $f$ may not be integrable on $\BR^d$.

Using an adaptation of the minimizing movements scheme introduced in \cite{JKT21}, we establish the following well-posedness result.

\begin{theorem}[Well-posedness, a summary of Propositions \ref{prop:existence}, \ref{prop:complementarity}, and \ref{prop: uniqueness} and Remark \ref{rmk: compact support}]\label{thm:wellposedness}
Let $n_0\in L^{\infty}(\R^d)\cap BV(\BR^d)$.
Let $\rho_0\in BV(\RR^d)$ be compactly supported, such that $\rho_0\in [0,1]$ almost everywhere.
Then for given $b, D\geq 0$, and any $T>0$, there exists a unique weak solution $(\rho, p, n)$ of \eqref{eqn: P} in $Q_T = \RR^d\times [0,T]$ in the sense of Definition \ref{def: weak solution}.

Moreover, $\rho$ and $p$ are compactly supported in $Q_T$, and satisfy the complementarity relation in the distribution sense:
$$
p(\Delta p+n-b)=0 \hbox { in } \mathcal{D}' (Q_T).
$$

If, additionally, $\rho_0(x)\in \{0,1\}$ almost everywhere in space, then for every time $t\in [0,T]$ we have $\rho\in \{0,1\}$ almost everywhere in space.
\end{theorem}

As we mentioned above, we obtain this result via a variational approximation scheme, which is a simplified version of the more general scheme introduced in \cite{JKT21}.  Although there are many different ways to establish the well-posedness of the system (such as the degenerate diffusion approach considered in \cite{PQV} and \cite{GKM}), we emphasize our variational scheme, as it has a very efficient numerical implementation via the Back-and-Forth method \cite{bfm, bfm_gf}.   For instance, the images displayed in Figure \ref{fig:death_dendrites} are computed on a high-resolution $1024\times 1024$ grid.  Carrying out simulations of this size has been out of reach for previous methods, owing to the difficult nonlinearities in the equation.  The efficiency of the scheme will be studied further in the upcoming paper \cite{jacobs_lee}.  In addition to the favorable numerics, the scheme also preserves many desirable properties of the true system, such as the patch-preserving property and various important estimates that will prove useful in our analysis of the scheme.

\subsection{Contraction and stability estimates}
\label{sec: comparison}

The rest of the results focus on the case $b=D=0$.
In this setting, it will be useful to introduce the quantity $\eta: = n_0-n$ which represents the amount of nutrient that has been consumed.
It solves
\[
\partial_t \eta = (n_0-\eta)\rho,\quad \eta(x,0) = 0,
\]
which can be derived from \eqref{eqn: P}, and which is equivalent to the $n$-equation.
Therefore, in what follows, in the case of $b = D = 0$, we will use $(\rho,p,n)$ and $(\rho,p,\eta)$ interchangeably as the solution of \eqref{eqn: P}.
In this case, the system \eqref{eqn: P} enjoys a lot of nice properties, especially for patch solutions, i.e., $\rho_0$ and $\rho$ take the value $0$ or $1$ almost everywhere in space at all times.

We first state the following $L^1$-stability estimate for patch solutions, from which the comparison principle can be readily derived.

\begin{theorem}[$L^1$-contraction, a simplified version of Theorem~\ref{thm:contraction}]
Suppose that $(\rho^i, p^i, \eta^i)$ $(i=0,1)$ are weak solutions of \eqref{eqn: P}, starting from initial datum $(\rho^i_0, n^i_0)$ respectively.
If $\rho_0^i$ are patches and $n^0_0\in L^{\infty}(\RR^d)$, then
\begin{equation*}
\norm{(\rho^1(\cdot,t)-\rho^0(\cdot,t))_+}_{L^1}\leq N(t)\norm{(n^1_0-n^0_0)_+}_{L^1}+M(t)\norm{(\rho_0^1-\rho_0^0)_+}_{L^1}.
\end{equation*}
Here $L^r= L^r(\RR^d)$,
\beq
N(t)
=
\begin{cases}
\frac{e^{(\norm{n^0_0}_{L^{\infty} }-1)t}-1}{\norm{n^0_0}_{L^{\infty} }-1} & \textup{if} \;\norm{n^0_0}_{L^{\infty} }\neq 1,\\
t &\textup{otherwise},
\end{cases}
\label{constant N}
\eeq
and
\beq
M(t)=
\begin{cases}
\frac{\norm{n^0_0}_{L^{\infty} }e^{(\norm{n^0_0}_{L^{\infty} }-1)t}-1}{\norm{n^0_0}_{L^{\infty} }-1} & \textup{if} \;\norm{n^0_0}_{L^{\infty} }\neq 1,\\
t+1 &\textup{otherwise}.
\end{cases}
\label{constant M}
\eeq

\end{theorem}

Based on the $L^1$-contraction and symmetries of the Laplacian operator, we can derive the following $BV$-estimates for $\rho$.  Note that here we consider a somewhat non-standard norm that we call the $BV_A$-norm, where $A$ is an antisymmetric matrix.  The $BV_A$-norm is the $L^1$-norm of the product $\nabla \rho\cdot Ax$.  For antisymmetric matrices $A$, $\nabla \rho\cdot Ax$ picks up a non-radial component of the derivative of $A$, with a stronger weight as one moves further from the origin.  By bounding this norm, we obtain a stronger control (compared to the vanilla $BV$-norm) on the behavior of the non-radial components of the boundary variation as the tumor grows.


\begin{theorem}[$BV$-estimates, Proposition \ref{prop: BV estimates}]
Given an antisymmetric matrix $A\in \RR^{d\times d}$ and some $g\in L^1(\RR^d)$, define
\[
\norm{g}_{BV_A(\RR^d)}:=\sup_{\vp\in C^{\infty}_c(\RR^d), \norm{\vp}_{L^{\infty}(\RR^d)}\leq 1}\int_{\RR^d} g(x) Ax\cdot \nabla \vp\,dx.
\]
Let $(\rho,p,\eta)$ be a weak solution of \eqref{eqn: P} with initial data $(\rho_0, n_0)$.
If $\rho_0\in L^1(\BR^d)\cap BV(\RR^d)$ is a patch, and $n_0\in L^{\infty}(\RR^d)\cap BV(\RR^d)$, then, with $N(t)$ and $M(t)$ given above in \eqref{constant N} and \eqref{constant M},
\begin{align*}
\norm{\rho(\cdot,t)}_{BV(\RR^d)}\leq &\;
N(t) \norm{n_0}_{BV(\RR^d)}+M(t)\norm{\rho_0}_{BV(\RR^d)},\\
\norm{\rho(\cdot,t)}_{BV_A(\RR^d)}\leq &\;
N(t)\norm{n_0}_{BV_A(\RR^d)}+M(t)\norm{\rho_0}_{BV_A(\RR^d)}.
\end{align*}
\end{theorem}

From the comparison principle as well as analysis on radial solutions, one can conclude the following result on the long-time behavior of the solution when $\|n_0\|_{L^\infty} < 1$.
In this case, the tumor eventually stops growing and approaches a stationary solution.

\begin{theorem}[Theorem \ref{thm: convergence when n_0 is less than 1}]
\label{thm: convergence when n_0 is less than 1 intro}
Suppose that $\norm{n_0}_{L^{\infty}(\RR^d)}<1$ and $\rho_0(x)\in \{0,1\}$.
Let $(\rho, p,\eta)$ be as given above.
Then
\begin{align*}
\norm{\rho(\cdot,t)-\rho_{\infty}}_{L^1(\RR^d)}
\leq &\; \frac{\norm{n_0\rho_0}_{L^1(\RR^d)}}{1-\norm{n_0}_{L^{\infty}(\RR^d)}} e^{(\norm{n_0}_{L^{\infty}(\RR^d)}-1)t},\\
\norm{\eta(\cdot,t)-n_0\rho_{\infty}}_{L^1(\RR^d)}\leq &\; \frac{\norm{n_0\rho_0}_{L^1(\RR^d)}}{1-\norm{n_0}_{L^{\infty}(\RR^d)}} e^{(\norm{n_0}_{L^{\infty}(\RR^d)}-1)t},
\end{align*}
where $\rho_{\infty}$ solves the elliptic equation
\[
(1-n_0)\rho_{\infty}-\Delta w_{\infty}=\rho_0, \quad w_{\infty}(1-\rho_{\infty})=0.
\]
\end{theorem}


\subsection{Geometry and regularity of the patch boundary}
Under suitable assumptions on the initial patch data $\rho_0$ and the initial nutrient $n_0$, we can study boundary regularity of the growing part of the set $\{\rho(\cdot,t)=1\}$.
We will use reflection invariance of the problem and comparison principle, using reflection-based geometry of the sets, to achieve these. Such an argument was used first in \cite{FK} and later in \cite{KK20, KKP} to obtain regularity results for interface motions with reflection invariance.

\medskip

The following version of the comparison principle plays an important role.

\begin{theorem}[Reflection comparison, Proposition \ref{prop: reflection comparison}]
Suppose that $(\rho, p, \eta)$ is a solution of \eqref{eqn: P} starting from the initial data $(\rho^0, n^0)$.  Given a hyperplane $H$, let $\rho_H, \rho^0_H, \eta_H$, and $n^0_H$ denote the reflections of $\rho, \rho^0, \eta$, and $n^0$ about the hyperplane $H$ respectively. Let $H^+$ be one of the half spaces generated by $H$.  If $\rho^0_H\leq \rho^0$ and $n^0_H\leq n^0$ almost everywhere in $H^+$, then $\rho_H\leq \rho$ and $\eta_H\leq \eta$ almost everywhere in $H^+$.
\end{theorem}

We say a set $S\subset\RR^d$ satisfies {\it $r$-reflection} if it contains the ball $B_r(0)= \{|x| < r \}$, and if for any hyperplane $H$ not intersecting with $B_r(0)$, the reflected image of $S$ with respect to $H$ is a subset of $S$ when restricted on the side of $H$ that contains the origin; see Definition \ref{def: set reflection}.
It is a notion that is stronger than being star-shaped, tailored to work with the reflection comparison above. With this concept, we obtain the following results based on the above theorem.

\begin{theorem}[Corollary \ref{cor:1}]\label{thm:reg}
Suppose $n_0$ is $C^1$, and its super-level sets satisfy $r$-reflection for some $r>0$.
Let $\Omega_0$ be an open bounded set in $\R^d$ contained in $B_r(0)$, and $\rho_0 = \chi_{\Omega_0}$.
Let $T(R)$ be defined as in Corollary \ref{star:shaped}.
Then the followings hold for any $0<\alpha<1$ and a dimensional constant $c_d$:
 \begin{enumerate}[(a)]
\item If $n_0(x)\geq 1$ on $\BR^d$, then for any $R>c_d r$, $\partial\{w(\cdot,t)>0\}$ is uniformly $C^{1,\alpha}$ in a unit neighborhood for any finite time range within $[T(R), \infty)$.
\item If $n_0(x)<1$ on $\BR^d$, the same holds in for any finite time range within  $[T(R), \infty)$ if $R>c_dr$ satisfies $B_R(0)\subset \{w_{\infty}>0\}$.
    Here $w_\infty$ is defined in Theorem \ref{thm: convergence when n_0 is less than 1 intro} above (or equivalently, Theorem \ref{thm: convergence when n_0 is less than 1}).
\end{enumerate}
\end{theorem}

In both cases of the above theorem, since we can start with $\rho_0 = \chi_{\Omega_0}$ where $\Omega_0$ is quite arbitrary, the evolution of the set $\Omega_t:=\{\rho(\cdot,t)=1\}$ may go through topological singularities such as merging of the free boundaries.
However, the above results state that, if the initial nutrient $n_0$ is ``well-prepared" outside of $B_r$, then after a finite time there is no further topological changes in the evolution, and $\Omega_t$ evolves with smooth boundary outside of $B_r$.

\subsection{The case of constant $n_0$}
When $n_0$ is constant, even stronger characterization can be provided for the evolution of patch solutions $\rho$.
In particular, the $\rho$-evolution in \eqref{eqn: P} coincides with density evolutions in some nutrient-free and parameter-free Hele-Shaw-type systems, up to explicit rescaling in space and time.
Such very surprising relation cannot be derived from trivial change of variables, but it crucially relies on special properties of the system \eqref{eqn: P} under the given assumptions.

\begin{theorem}[Two master dynamics, Lemma \ref{lem: general test function}, Proposition \ref{prop: master dynamics} and Proposition \ref{prop: master dynamics 2}]\label{thm:master_dynamics}
Let $(\rho, p, n)$ solve \eqref{eqn: P} with $\rho_0(x)\in \{0,1\}$ almost everywhere and $n_0(x)= n_0>0$ being constant.
Then $n$ can be represented in terms of $\rho$ using \eqref{eqn: n solution formula}.
Let $m(t)$ be explicitly defined by \eqref{eqn: volume}.
Then the total mass of the tumor satisfies
\[
\int_{\BR^d} \rho(x, t)\, dx = m(t)\int_{\BR^d} \rho_0(x)\, dx.
\]
Moreover,
\begin{enumerate}
\item
Let $(\rho_*,p_*)$ be a weak solution of
\beq\label{eqn: HS revisited}
\partial_t \rho_* -\nabla\cdot (\rho_*\nabla p_* ) = \rho_0,\quad \rho_*\leq 1,\quad p_*\in P_\infty(\rho_*),\quad \rho_*|_{t = 0} = \rho_0.
\tag{HS}
\eeq
Then $\rho(x,t) = \rho_* (x, m(t)-1 )$ for all $t\geq 0$.
\item
Let $(\rho_\dag, p_\dag)$ be a weak solution of
\beq
\label{eqn: evolution under potential}
\partial_t \rho_\dag -\nabla \cdot \Big[\rho_\dag \nabla \big( p_\dag +V(x)\big)\Big] = 0,\quad \rho_\dag \leq 1,\quad p_\dag\in P_\infty(\rho_\dag),\quad \rho_\dag|_{t=0} = \rho_0,
\tag{HS'}
\eeq
where $V(x) = \frac{|x|^2}{2d}$.
Then
$\rho(x,t) = \rho_\dag ( m(t)^{-\f1d} x,\, \ln m(t) )$ for all $t\geq 0$.
\end{enumerate}
Here the notion of weak solutions is provided in Definition \ref{def: weak solution p equation} below.
\end{theorem}


Thanks to the master dynamics, we can characterize the long-time behavior of the patch solutions when $n_0>0$ is constant in $\BR^d$ (c.f.~Theorem \ref{thm: convergence when n_0 is less than 1 intro}) under a suitable rescaling.

\begin{theorem}[Proposition \ref{prop: long-time asymptotics for constant n_0 no less than 1}]
Suppose $n_0>0$ is constant in $\BR^d$.
Let $m(t)$ be defined in \eqref{eqn: volume}.
Assume $\Omega_0$ to be a bounded open set, such that $B_{r_1}(0)\subset \Omega_0 \subset B_{r_2}(0)$ for some $r_1,r_2>0$.
Let $r_\infty>0$ be defined such that $|B_{r_\infty}(0)| = |\Omega_0|$.

Let $(\rho,p,n)$ solve \eqref{eqn: P} with $\rho_0 = \chi_{\Omega_0}$.
Then there exists a constant $C>0$ only depending on $r_1$ and $r_2$, but not on $n_0$, such that
\beqo
W_2\Big( \rho\big(m(t)^{\f1d} x,t\big), \chi_{B_{r_\infty}(0)}(x) \Big)\leq C m(t)^{-\f1d}
\eeqo
for all $t\geq 0$.
Here $W_2$ denotes the 2-Wasserstein distance.
\end{theorem}

Lastly, we present a regularity result for the rescaled solutions when $n_0>1$. It is possible to use viscosity solutions approach to study \eqref{eqn: HS revisited} and \eqref{eqn: evolution under potential} as pressure-driven free boundary problems.
For instance, for \eqref{eqn: HS revisited} we have
$\rho = \chi_{\{p>0\}}$ and the set $\{p>0\}$ evolves according to
$$
\left\{\begin{array}{lll}
-\Delta p =\chi_{\Omega_0}&\hbox{ in }\{p>0\},\\
V_n = |\nabla p|&\hbox{ on }\partial\{p>0\}.
\end{array} \right.
$$
In the case of a classic Hele-Shaw problem, the interior pressure equation is replaced by $-\Delta p =0$ on a perforated domain, with the value of $p$ being prescribed along the fixed inner boundary.
In such case, the free boundary regularity has been studied in \cite{CJK07}. While we expect parallel results to hold for our problem, it seems not straightforward to verify this.

\begin{theorem}[Theorem \ref{thm: uniform regularity}]
Fix $n_0>1$, and let $\rho_0$, $r_1$, $r_2$, and $m(t)$ be as in above theorem.
Then there is $\alpha\in (0,1)$ and $T>0$ which depends on $r_1$, $r_2$, and $n_0$, such that the followings hold for all $t \geq T$.
\begin{enumerate}[(a)]
\item
The rescaled set $\tilde{\Omega}_t:= m(t)^{-1/d}\Omega_t$ has uniformly $C^{1,\alpha}$-boundary;
\item
The rescaled nutrient variable $\tilde{n}(x,t):= n(m^{1/d}(t)x,t)$ is uniformly bounded in $C^{\alpha}(\{|x|\geq 2m^{-1/d}(t)r_2\})$.
\end{enumerate}
\end{theorem}

\section{Discrete-in-time Scheme and Well-posedness}
\label{sec: discrete in time scheme}

In this section, we explicitly construct solutions for the PDE \eqref{eqn: P} under the assumption that $n_0\in L^{\infty}(\RR^d)\cap BV(\RR^d)$ and $\rho_0\in L^1(\BR^d)\cap BV(\RR^d)$ along with $\rho_0\in [0,1]$ almost everywhere.
We shall use the following discrete-in-time scheme introduced in \cite{JKT21}. Given a time step $\tau>0$, an initial density $\rho^{0,\tau} = \rho_0$, and an initial nutrient density $n^{0,\tau} = n_0$, we iterate that
\begin{align}
\rho^{k+1, \tau}= &\;\argmin_{\rho\leq 1} \frac{1}{2\tau}W_2^2\Big( \rho, \rho^{k,\tau}\big(1+\tau (n^{k,\tau}-b)\big)\Big), \label{eq:rho_update}\\
n^{k+1,\tau} = &\; e^{\tau D\Delta}\big(n^{k,\tau}(1-\tau\rho^{k+1,\tau})\big),\nonumber
\end{align}
where $e^{\tau D\Delta}$ is the heat kernel.
The optimal pressure variable can be recovered by considering the dual problem to \eqref{eq:rho_update},
\begin{equation}\label{eq:p_update}
p^{k+1,\tau}=\argmax_{p\geq 0} \int_{\RR^d} p^{c_{\tau}}\rho^{k,\tau}(1+\tau (n^{k,\tau}-b))-p\,dx,
\end{equation}
where
\[
p^{c_{\tau}}(x)=\inf_{y\in \RR^d} p(y)+\frac{1}{2\tau}|y-x|^2,
\]
is the $c$-transform.
We additionally define $p^{0,\tau} = 0$.

We now import the following three crucial lemmas from \cite{JKT21} with minor adaptations.
The first lemma gives a link between the primal and dual variables; the second lemma establishes an energy dissipation property for the scheme; and the final lemma establishes some useful properties enjoyed by the discrete density.

\begin{lemma}[\cite{JKT21}]\label{lem:jkt1}
The optimal primal and dual variables $\{\rho^{k,\tau}\}_k$ and $\{p^{k,\tau}\}_k$ are linked through the following relations
\begin{equation}\label{eq:pressure_density_duality}
p^{k+1,\tau}(1-\rho^{k+1,\tau})=0, \quad \big(\id+\tau \nabla p^{k+1,\tau}\big)_{\#} \rho^{k+1,\tau}=\rho^{k,\tau}(1+\tau (n^{k,\tau}-b)).
\end{equation}

\end{lemma}
\begin{lemma}[\cite{JKT21}]\label{lem:jkt2}
Each step of the scheme enjoys the energy dissipation property
\begin{equation}\label{eq:edi}
\frac{1}{2}\norm{\nabla p^{k+1,\tau}}_{L^2(\RR^d)}^2\leq \int_{\RR^d} \rho^{k,\tau}(n^{k,\tau}-b)p^{k+1,\tau}\,dx.
\end{equation}
\end{lemma}

\begin{lemma}[\cite{JKT21}]\label{lem:density_01}\label{lem:jkt3}
For almost every $x\in \RR^d$, we have $\rho^{k,\tau}(1-\tau b)\leq \rho^{k+1,\tau}$.  Furthermore, if $\rho_0(x)\in \{0,1\}$ almost everywhere and $b=0$, then $\rho^{k,\tau}(x)\in \{0,1\}$ almost everywhere.
\end{lemma}

Now we are ready to  introduce piecewise-constant-in-time interpolants: for $t\in [k\tau, (k+1)\tau)$, define
\[
\rho^{\tau}(x,t):=\rho^{k,\tau}(x),\quad
n^{\tau}(x,t):=n^{k,\tau}(x),\quad
p^{\tau}(x,t):=p^{k,\tau}(x).
\]
In addition, for $t\leq 0$, define $\rho^{\tau}(x,t)=\rho_0(x)$ and $n^{\tau}(x,t)=n_0(x)$.
Unlike in \cite{JKT21}, our growth rate here is independent of the pressure.
Hence, we will need the following estimates to obtain compactness for the interpolants.
\begin{lemma}\label{lem:scheme_estimates}
Let $\rho^{\tau}, p^{\tau}, n^{\tau}$ be the discrete interpolants defined above. For any $t\geq 0$, we have
\[
\norm{\rho^{\tau}(\cdot,t)}_{L^1(\RR^d)}\leq e^{t\norm{n_0}_{L^{\infty}(\RR^d)}}\norm{\rho_0}_{L^1(\RR^d)} =: B(t),
\]
and
\begin{align*}
\norm{p^{\tau}}_{L^2(\RR^d\times [0,t])}\leq &\;
C 
B(t)^{\frac{d+4}{2d}}
\norm{n_0}_{L^{\infty}(\RR^d)}^{\f12},\\
\norm{p^{\tau}}_{L^1(\RR^d\times [0,t])}\leq &\; C 
B(t)^{\frac{d+2}{d}},\\
\norm{\nabla p^{\tau}}_{L^2(\RR^d\times [0,t])}\leq &\;  
C B(t)^{\frac{d+2}{2d}} \norm{n_0}_{L^{\infty}(\RR^d)}^{\f12}.
\end{align*}
Here $C>0$ is a universal constant only depending on $d$.
Moreover, when $\tau\ll 1$ such that $\tau b < 1$,
\[
\norm{\nabla \rho^\tau (\cdot,t)}_{L^1(\RR^d)}+\norm{\nabla n^\tau(\cdot,t)}_{L^1(\RR^d)}\leq e^{(2\norm{n_0}_{L^{\infty}(\RR^d)}+1)t}\Big(\norm{\nabla \rho_0}_{L^1(\RR^d)}+\norm{\nabla n_0}_{L^1(\RR^d)}\Big).
\]
\end{lemma}
\begin{proof}
It is clear that $\|n^{\tau}(\cdot,t)\|_{L^\infty}$ is non-increasing with respect to time and thus
\[
\norm{\rho^{\tau}(\cdot,t)}_{L^1(\RR^d)}\leq \big(1+\tau\norm{n_0}_{L^{\infty}(\RR^d)}\big)
\norm{\rho^{\tau}(\cdot,t-\tau)}_{L^1(\RR^d)}.
\]
Iterating and using the fact that $(1+\tau\norm{n_0}_{L^{\infty}(\RR^d)})^k\leq e^{\tau k\norm{n_0}_{L^{\infty}(\RR^d)}}$ for any $k$, the first result follows.

For the second result, we can use the duality relation $p^{\tau}(1-\rho^{\tau})=0$ to obtain
\[
\norm{p^{\tau}(\cdot,t)}_{L^1(\RR^d)}\leq \norm{\rho^{\tau}(\cdot,t)}_{L^2(\RR^d)}\norm{p^{\tau}(\cdot,t)}_{L^2(\RR^d)}
\leq \norm{\rho^{\tau}(\cdot,t)}_{L^1(\RR^d)}^{\frac{1}{2}}\norm{p^{\tau}(\cdot,t)}_{L^2(\RR^d)}.
\]
Next, the Gagliardo-Nirenberg inequality implies that
\[
\norm{p^{\tau}(\cdot,t)}_{L^2(\RR^d)}\leq C\norm{p^{\tau}(\cdot,t)}_{L^1(\RR^d)}^{\frac{2}{d+2}}\norm{\nabla p^{\tau}(\cdot,t)}_{L^2(\RR^d)}^{\frac{d}{d+2}},
\]
where $C>0$ is a universal constant only depending on $d$.
Thus,
\[
\norm{p^{\tau}(\cdot,t)}_{L^2(\RR^d)}\leq C \norm{\rho^{\tau}(\cdot,t)}_{L^1(\RR^d)}^{\frac{1}{d}}\norm{\nabla p^{\tau}(\cdot,t)}_{L^2(\RR^d)}.
\]
Integrating in time and combining this with the energy dissipation inequality \eqref{eq:edi}, we see that
\[
\norm{p^{\tau}}_{L^2(\RR^d\times [0,t])}^2\leq C B(t)^{\frac{2}{d}}
\norm{\nabla p^{\tau}}_{L^2(\RR^d\times [0,t])}^2
\leq
C B(t)^{\frac{2}{d}}
\norm{ p^{\tau}}_{L^1(\RR^d\times [0,t])}
\norm{n_0}_{L^{\infty}(\RR^d)}.
\]
Here we used the fact that
\[
\norm{\nabla p^{k+1,\tau}}_{L^2(\RR^d)}^2
\leq
2\int_{\RR^d}\rho^{k,\tau}n^{k,\tau}p^{k+1,\tau}\,dx
\leq 2
\norm{ p^{k+1,\tau}}_{L^1(\RR^d)} \norm{n_0}_{L^{\infty}(\RR^d)}.
\]
Reusing the estimate on the pressure $L^1$-norm from above and noticing that
\[
\int_0^t B(s)\,ds \leq \|n_0\|_{L^\infty(\BR^d)}^{-1} B(t),
\]
we get
\[
\norm{p^{\tau}}_{L^2(\RR^d\times [0,t])}
\leq C 
B(t)^{\frac{d+4}{2d}}
\norm{n_0}_{L^{\infty}(\RR^d)}^{\f12}.
\]
Combining this with our previous work, we get all of the estimates except for the last one.

To prove the final estimate, we can use the $BV$-bound from \cite{bv_ot} to obtain that, when $\tau b< 1$,
\[
\begin{split}
\norm{\nabla \rho^{k+1,\tau}}_{L^1(\RR^d)}\leq &\;
\big\|\nabla \big(\rho^{k,\tau}(1+\tau (n^{k,\tau}-b) )\big)\big\|_{L^1(\RR^d)}\\
\leq &\; \norm{\nabla \rho^{k,\tau}}_{L^1(\RR^d)} +\tau \norm{n_0}_{L^{\infty}(\RR^d)} \norm{\nabla \rho^{k,\tau}}_{L^1(\RR^d)}+ \tau\norm{\nabla n^{k,\tau}}_{L^1(\RR^d)}.
\end{split}
\]
It is then straightforward to see that the interpolants satisfy
\[
\norm{\nabla \rho^{\tau}(\cdot,t)}_{L^1(\RR^d)}
\leq
\norm{\nabla \rho_0}_{L^1(\RR^d)} + \norm{n_0}_{L^{\infty}(\RR^d)} \norm{\nabla \rho^{\tau}}_{L^1(\RR^d\times [0,t])}+\norm{\nabla n^{\tau}}_{L^1(\RR^d\times [0,t])}.
\]
From the discrete scheme, we also have
\[
\begin{split}
\norm{\nabla n^{k+1,\tau}}_{L^1(\RR^d)}\leq &\;
\big\|\nabla e^{\tau D\Delta} \big(n^{k,\tau}(1-\tau\rho^{k,\tau})\big)\big\|_{L^1(\RR^d)}
\leq
\big\|\nabla \big(n^{k,\tau}(1-\tau\rho^{k,\tau})\big)\big\|_{L^1(\RR^d)}\\
\leq &\; \norm{\nabla n^{k,\tau}}_{L^1(\RR^d)}+\tau\norm{n_0}_{L^{\infty}(\RR^d)}\norm{\nabla \rho^{k,\tau}}_{L^1(\RR^d)}.
\end{split}
\]
Thus, the interpolants satisfy
\[
\norm{\nabla n^{\tau}(\cdot,t)}_{L^1(\RR^d)}\leq
\norm{\nabla n_0}_{L^1(\RR^d)} +  \norm{n_0}_{L^{\infty}(\RR^d)}\norm{\nabla \rho^{\tau}}_{L^1(\RR^d\times [0,t])}. 
\]
Summing the two estimates together, we see that
\[
\begin{split}
&\; \norm{\nabla \rho^{\tau}(\cdot,t)}_{L^1(\RR^d)}+\norm{\nabla n^{\tau}(\cdot,t)}_{L^1(\RR^d)}\\
\leq &\; \norm{\nabla \rho_0}_{L^1(\RR^d)}+\norm{\nabla n_0}_{L^1(\RR^d)} \\
&\; + \big(2\norm{n_0}_{L^{\infty}(\RR^d)}+1\big)\Big(\norm{\nabla \rho^{\tau}}_{L^1(\RR^d\times [0,t])}+\norm{\nabla n^{\tau}}_{L^1(\RR^d\times [0,t])}\Big).
\end{split}
\]
The final claimed inequality now follows by Gronwall's inequality.
\end{proof}

In addition to the above pressure estimates, we have the following control on the time derivatives of the density.

\begin{lemma}\label{lem:rho_time_lipschitz}
Let $\rho^{\tau}$ be the discrete density interpolant defined above, with $\tau b<1$.
For any $0\leq t_0 <t_1\leq T$, we have
\[
\norm{\rho^{\tau}(\cdot,t_1)-\rho^{\tau}(\cdot, t_0)}_{L^1(\RR^d)}\leq  (t_1-t_0+\tau)\Big(2b+\norm{n_0}_{L^{\infty}(\RR^d)}\Big)\norm{\rho^{\tau}}_{L^{\infty}([0,T];L^1(\RR^d))}.
\]
\end{lemma}
\begin{proof}
Using the fact that for any $k\geq 0$ we have $\rho^{k+1,\tau}\geq (1-\tau b)\rho^{k,\tau}$, we can estimate
\[
\begin{split}
&\;\norm{\rho^{\tau}(\cdot,t_1)-\rho^{\tau}(\cdot, t_0)}_{L^1(\RR^d)}\\
\leq  &\; \int_{\RR^d} \rho^{\tau}(x,t_1)+\rho^{\tau}(x,t_0)-2(1-\tau b)^{\lceil\frac{1}{\tau}(t_1-t_0)\rceil}\rho^{\tau}(x,t_0)\, dx\\
\leq &\; 2\left[1-(1-\tau b)^{\lceil\frac{1}{\tau}(t_1-t_0)\rceil}\right] \norm{\rho^{\tau}(\cdot, t_0)}_{L^1(\RR^d)}+\tau \sum_{k=\lfloor t_0/\tau \rfloor}^{\lfloor t_1/\tau \rfloor-1} \int_{\RR^d} \rho^{k,\tau}n^{k,\tau}\,dx \\
\leq &\; (t_1-t_0+\tau)\Big(2b+\norm{n_0}_{L^{\infty}(\RR^d)}\Big)\norm{\rho^{\tau}}_{L^{\infty}([0,T];L^1(\RR^d))}.
\end{split}
\]
\end{proof}

Using the estimates above, we can prove that the discrete interpolants converge to a weak solution of the continuum PDE as we send $\tau \to 0$.

\begin{prop}\label{prop:existence}
Assume $n_0\in L^{\infty}(\RR^d)\cap BV(\RR^d)$ and $\rho_0\in L^1(\BR^d)\cap BV(\RR^d)$ along with $\rho_0\in [0,1]$ almost everywhere.
Take an arbitrary finite $T>0$ and denote $Q_T = \BR^d\times [0,T]$.
The family $\{\rho^{\tau}\}_{\tau>0}$ is strongly $L^1(Q_T)$ precompact,
$\{p^{\tau}\}_{\tau>0}$ is weakly precompact in $L^2([0,T];H^1(\BR^d))$,
and $\{n^{\tau}\}_{\tau>0}$ is precompact in the weak-$*$ topology of $L^\infty(Q_T)$.
Let $(\rho, p,n)$ be a limit point in the above-mentioned topology.
Then $(\rho, p,n)$ is a weak solution of the tumor growth PDE \eqref{eqn: P}.
More precisely, we have $\rho \in L^1(Q_T)$, $p\in L^2([0,T];H^1(\BR^d))$, and $n\in L^\infty(Q_T)$, satisfying that
for any $\psi\in H^1(Q_T)$ such that $\psi(\cdot,T)=0$ almost everywhere, it holds
\beq
\int_0^T\int_{\RR^d}\nabla \psi\cdot \nabla p- \rho\partial_t \psi \,dx \,dt =\int_{\RR^d} \psi(x,0)\rho_0(x)\,dx +\int_0^T\int_{\RR^d} \psi \rho (n-b)\,dx\,dt ,
\label{eqn: weak formulation}
\eeq
and
\[
\partial_t n-D\Delta n=-\rho n \mbox{ in }\mathcal{D}'(Q_T),\quad  n(x,0) = n_0(x).
\]
Moreover, $\rho\in C^{0,1}([0,T];L^1(\BR^d))\cap L^\infty([0,T];BV(\BR^d))$, $n\in L^\infty([0,T];BV(\BR^d))$, and $\rho\in [0,1]$, $p\geq 0$ with $p(1-\rho)=0$  a.e. on $Q_T$.
If $\rho_0\in \{0,1\}$ and $b=0$, then for all $t\in [0,T]$ we have $\rho\in \{0,1\}$ a.e. in $\RR^d$.

\end{prop}
\begin{proof}
From the estimates in Lemma \ref{lem:scheme_estimates}, $\{p^\tau\}_{\tau>0}$ is weakly precompact in $L^2([0,T];H^1(\BR^d))$. 
%
%
Thanks to Lemma \ref{lem:rho_time_lipschitz}, for any $0\leq t_0\leq t_1\leq T$,
\beq
\limsup_{\tau\to 0} \norm{\rho^{\tau}(\cdot,t_1)-\rho^{\tau}(\cdot, t_0)}_{L^1(\RR^d)}\leq (t_1-t_0)(2b+\norm{n_0}_{L^{\infty}(\RR^d)})B(T).
\label{eqn: Lipschitz in time}
\eeq
Combining this with Lemma \ref{lem:scheme_estimates}, we know that $\{\rho^\tau\}_{\tau>0}$ is uniformly bounded and equicontinuous in $L^1(Q_T)$ (in space-time).
Thus, by the Riesz-Fr\'{e}chet-Kolmogorov compactness theorem, $\{\rho^{\tau}\}_{\tau>0}$ is strongly $L^1(Q_T)$ precompact.

With these compactness properties, now we turn to considering the PDE.
Given a smooth test function $\psi$ that vanishes at time $T$, we can use \eqref{eq:pressure_density_duality} to obtain
\[
\begin{split}
&\; \int_{\RR^d} \frac{\rho^{\tau}(x, t)-\rho^{\tau}(x,t-\tau)}{\tau}\psi(x,t)\,dx\\
= &\; \int_{\RR^d} \rho^{\tau}(x,t)\frac{\psi(x,t)-\psi(x+\tau\nabla p^{\tau}(x,t),t)}{\tau}+(n^{\tau}(x,t-\tau)-b)\rho^{\tau}(x,t-\tau)\psi(x,t)\,dx.
\end{split}
\]
Integrating both sides in time along $[\tau, T]$, we get
\[
\begin{split}
&\; \int_{\tau}^{T-\tau}\int_{\RR^d} \rho^{\tau}(x,t)\frac{\psi(x,t)-\psi(x,t+\tau)}{\tau}\,dx\,dt\\
&\;+\frac{1}{\tau}\int_0^{\tau}\int_{\RR^d}\rho^{\tau}(x,t+T-\tau)\psi(x,t+T-\tau)-\rho^{\tau}(x,t)\psi(x,t+\tau)\,dx\,dt\\
= &\; \int_{\tau}^T\int_{\RR^d} \rho^{\tau}(x,t)\frac{\psi(x,t)-\psi(x+\tau\nabla p^{\tau}(x,t),t)}{\tau}+(n^{\tau}(x,t-\tau)-b)\rho^{\tau}(x,t-\tau)\psi(x,t)\,dx\,dt.
\end{split}
\]
Thanks to the smoothness of $\psi$ and the estimates from above, the previous line is equivalent to
\begin{equation}\label{eq:approx_pde}
\int_{0}^{T}\int_{\RR^d} -\rho^{\tau}\partial_t \psi\,dx\,dt
-\int_{\RR^d} \rho_0\psi(x,0)\,dx
=
\epsilon_{\tau}+\int_{0}^T\int_{\RR^d} -\rho^{\tau}\nabla \psi\cdot\nabla p^{\tau}+(n^{\tau}-b)\rho^{\tau}\psi\,dx\,dt.
\end{equation}
Here
\[
\begin{split}
|\epsilon_{\tau}|\leq &\; C\sqrt{\tau}\|\nabla \psi\|_{L^\infty(Q_T)} \norm{\nabla p^{\tau}}_{L^2(Q_T)} \norm{\rho^\tau}_{L^{\infty}([0,T];L^1(\RR^d))}^{1/2}
\\
&\;+ C\tau\left[\norm{D^2\psi}_{L^{\infty}(Q_T)}\norm{\nabla p^{\tau}}_{L^2(Q_T)}^2
\right.\\
&\; \qquad \left.+
\big(\norm{\partial_t \psi}_{L^{\infty}(Q_T)}
+\norm{\partial_t^2 \psi}_{L^{\infty}(Q_T)}\big)
\norm{\rho^\tau}_{L^{\infty}([0,T];L^1(\RR^d))}\right],
\end{split}
\]
where $C$ is a universal constant depending on $\|n_0\|_{L^\infty}$, $b$, and $T$.
Clearly, $\lim_{\tau\to 0} |\epsilon_{\tau}|=0$.
Since $p^{\tau}\in L^2([0,T];H^1(\RR^d))$ and $p^{\tau}(1-\rho^{\tau})=0$, it follows that $\rho^{\tau}\nabla p^{\tau}=\nabla p^{\tau}$.

Now we can claim that, there exists $\rho \in L^1(Q_T)$, $p\in L^2([0,T];H^1(\BR^d))$, $n\in L^\infty(Q_T)$, and
%
%
%
a sequence $\{\tau_j\}_j$ converging to $0$, such that $\rho^{\tau_j}\to \rho\in [0,1]$ strongly in $L^1(Q_T)$, $p^{\tau_j}\rightharpoonup p\geq 0$ weakly in $L^2([0,T];H^1(\BR^d))$, and $n^{\tau_j}$ converges weak-$*$ in $L^\infty(Q_T)$ to $n$.
It is then clear that we can pass to the limit in \eqref{eq:approx_pde} to obtain \eqref{eqn: weak formulation}.
One may relax the regularity of $\psi$ to find \eqref{eqn: weak formulation} actually holds for all $\psi\in H^1(Q_T)$ satisfying $\psi(\cdot,T) = 0$ almost everywhere.

The strong convergence from $\rho^{\tau_j}$ to $\rho$ in $L^1(Q_T)$ also implies
\[
\lim_{j\to +\infty}\|\rho^{\tau_j}(\cdot,t)-\rho(\cdot,t)\|_{L^1(\BR^d)}= 0
\]
for almost every $t\in [0,T]$.
Hence, by \eqref{eqn: Lipschitz in time}, up to modifying $\rho(x,t)$ for those $t$ on a measure-zero set in $[0,T]$ if necessary, we can make $\rho$ be Lipschitz continuous in $L^1(\BR^d)$ with respect to time.
That $\rho,n\in L^\infty([0,T];BV(\BR^d))$ follows from Lemma \ref{lem:scheme_estimates}.

Thanks to the strong $L^1$-convergence of $\{\rho^{\tau_j}\}_j$, we can readily verify that the $n$-equation and $p(1-\rho) = 0$ hold in the sense of distribution (we omit the details).
Given the regularity of $p$ and $\rho$, $p(1-\rho) = 0$ almost everywhere in $Q_T$.
Hence, we may suitably modify $p$ on a measure-zero set of $Q_T$, if necessary, to achieve $p(1-\rho) = 0$ everywhere on $Q_T$.

By the strong $L^1$-convergence of $\{\rho^{\tau_j}\}_j$, Lemma \ref{lem:density_01}, and the fact that $\rho$ is Lipschitz in $L^1(\RR^d)$ with respect to time, we find that for all $t\in [0,T]$, $\rho\in \{0,1\}$ almost everywhere if $\rho_0\in \{0,1\}$ almost everywhere and $b=0$.
\end{proof}

Then we show that our solution satisfies the so-called complementarity condition.

\begin{prop}\label{prop:complementarity}
If $(\rho, p, n)$ is a weak solution of the tumor growth PDE \eqref{eqn: P} (in the sense of Definition \ref{def: weak solution}),
then for any bounded open set $\Omega\subset \RR^d$ and $\psi\in L^2([0,T];H^1(\Omega)\cap L^1(\RR^d))$,
such that $\psi(1-\rho)=0$ and $\psi|_{\partial \Omega\times [0,T]}=0$, we have
\[
\int_0^T\int_{\Omega}\nabla \psi\cdot \nabla p -\psi  (n-b)\,dx\,dt =0.
\]
\end{prop}
\begin{remark}
Let us note that the space of $\psi$ satisfying the above conditions is nontrivial.  For instance, given a smooth function $\eta:\RR^d\times [0,T]\to \BR$ such that $\eta$ is compactly supported inside $\Omega \times [0,T]$, the choice $\psi:=p\eta$ satisfies all of the above conditions.
\end{remark}
\begin{proof}
We begin by assuming that there exists $\delta>0$ such that $\psi(x,t)=0$ for all $(x,t)\in [T-\delta,T]\times\Omega$.
Fix $\epsilon\in (0,\delta)$ and let
\[
\psi^{\epsilon}(x,t)=\frac{1}{\epsilon}\int_{t}^{\min(T, t+\epsilon)} \psi(x,s)\, ds.
\]
Since $\psi^{\epsilon}\in H^1(\Omega\times [0,T])$, when extended by zero to the whole $Q_T$, it is a valid test function for the weak formulation (see \eqref{eqn: weak formulation}).
Hence, we have
\[
\int_{\Omega\times [0,T]} -\rho\partial_t \psi^{\epsilon}+\nabla \psi^{\epsilon}\cdot \nabla p-\psi^{\epsilon}\rho (n-b) \,dx\,dt =\int_\Omega \psi^{\epsilon}(x,0)\rho_0(x)\, dx.
\]
Note that
\[
\rho\partial_t\psi^{\epsilon}=\rho(x,t)\frac{\psi(x, \min(T, t+\epsilon))-\psi(x,t)}{\epsilon}\leq \frac{\psi(x, \min(T, t+\epsilon))-\psi(x,t)}{\epsilon}=\partial_t \psi^{\epsilon},
\]
where we use $\psi(1-\rho)=0$ and the non-negativity of $\psi$ to justify the inequality. Therefore,
 \[
\int_{\Omega\times [0,T]} \nabla \psi^{\epsilon}\cdot \nabla p-\psi^{\epsilon}\rho (n-b)\,dx\,dt \leq \int_\Omega \psi^{\epsilon}(x,0)\big(\rho_0(x)-1\big)\, dx\leq 0.
\]
Sending $\epsilon\to 0$ and once again using $\psi(1-\rho)=0$, we see that
\[
\int_{\Omega\times [0,T]} \nabla \psi\cdot \nabla p-\psi(n-b) \,dx\,dt \leq  0,
\]
giving us one side of the equation.

To obtain the other direction, we instead smooth backwards in time and define
\[
\psi_{\epsilon}(x,t)=\frac{1}{\epsilon}\int_{\max(0, t-\epsilon)}^{t} \psi(x,s)\, ds.
\]
Note that $\psi_{\epsilon}(x,T)=0$ since $\psi$ vanishes on $[T-\delta,T]$.
An analogous argument shows that $\rho\partial_t \psi_{\epsilon}\geq \partial_t\psi_{\epsilon}$, and thus
\[
\int_{\Omega\times [0,T]} \nabla \psi_{\epsilon}\cdot \nabla p-\psi_{\epsilon} (n-b) \,dx\,dt
\geq  \int_\Omega \psi_{\epsilon}(x,0)\big(\rho_0(x)-1\big)\, dx\geq 0.
\]
Note that $\psi_\epsilon(x,0)\leq 0$.
Sending $\epsilon\to 0$ and combining with our previous work, we obtain
\[
\int_{\Omega\times [0,T]} \nabla \psi\cdot \nabla p-\psi (n-b) \,dx\,dt =  0,
\]
for all non-negative $\psi$ such that $\psi(1-\rho)=0,$
$\psi|_{\partial \Omega\times [0,T]}=0$, and $\psi(x,t)=0$ on $\Omega\times [T-\delta,T]$.
Since the equation is linear in $\psi$ and does not include time derivatives, we can drop the non-negativity assumption and then take limits to drop the assumption that $\psi$ vanishes on $\Omega\times [T-\delta,T]$.
\end{proof}

We now proceed to show the uniqueness of weak solutions. Similar results have been established in the literature using the Hilbert duality principle \cite{PQV, GKM}. However, they require the nutrient variable to be at least $L^1_tH^1_x$, which does not hold in our case when $D = 0$. Instead, we proceed with $L^1$-contraction approach to provide a unified proof of uniqueness for all $D\geq 0$.

\medskip

Let us focus on the $\rho$-equation for the moment.
Consider the model problem
\beq
\rho_t - \nabla\cdot(\rho \nabla p) = f,\quad \rho\leq 1,\quad p\in P_\infty (\rho),\quad \rho(x,0) = \rho_0.
\label{eqn: model rho problem}
\eeq
Similar to Definition \ref{def: weak solution}, we introduce the notion of its weak solutions.

\begin{definition}\label{def: weak solution p equation}
Let $\rho_0\in L^1(\BR^d)\cap BV(\RR^d)$ such that $\rho_0\in [0,1]$ almost everywhere.
Fix $T>0$ and denote $Q_T = \BR^d\times [0,T]$.
Assume $f\in L^\infty([0,T];L^1(\BR^d))\cap L^\infty(Q_T)$.
Non-negative functions $\rho \in C([0,T]; L^1(\RR^d))$ and  $p\in L^2([0,T]; H^1(\RR^d))$, which are defined on $\BR^d\times [0,T]$, are said to form a weak solution of \eqref{eqn: model rho problem}, if they satisfy:
\begin{enumerate}[(i)]
\item $\rho\in [0, 1]$ and $p(1-\rho)=0$ in $Q_T$;
\item For any $\psi\in C_0^\infty(Q_T)$ that vanishes at $t=T$, we have
\[
\int_0^T\int_{\RR^d}\nabla \psi\cdot \nabla p- \rho\partial_t \psi \,dx\,dt =\int_{\RR^d} \psi(x,0)\rho_0\,dx
+\int_0^T\int_{\RR^d} \psi f\,dx\,dt.
\]
\end{enumerate}
\end{definition}

\begin{remark}\label{rmk: existence of solution of model problem}
Under the assumptions that $\rho_0\in L^1(\BR^d)\cap BV(\BR^d)$, and $f\in L^\infty([0,T];L^1(\BR^d))\cap L^\infty(Q_T)\cap L^\infty([0,T];BV(\BR^d))$, it is not difficult to show existence of weak solutions of \eqref{eqn: model rho problem}.
In fact, one may use still the discrete scheme (c.f.\;\eqref{eq:rho_update} and \eqref{eq:p_update})
\begin{align*}
\rho^{k+1, \tau}= &\;\argmin_{\rho\leq 1} \frac{1}{2\tau}W_2^2\Big( \rho, \rho^{k,\tau}+\tau f^{k,\tau}\Big),\\
p^{k+1,\tau}= &\; \argmax_{p\geq 0} \int_{\RR^d} p^{c_{\tau}}\big(\rho^{k,\tau}+\tau  f^{k,\tau}\big)-p\,dx,
\end{align*}
where $f^{k,\tau}$ can be defined as, e.g., time-average of $f$ on each small time interval of size $\tau$.
Then arguing as before, we can prove the existence.
We omit the details.
\end{remark}

We can show that the weak solution of \eqref{eqn: model rho problem} satisfies the comparison principle.
See Appendix \ref{appendix: comparison proof} for the proof, which follows the Hilbert duality argument in \cite{PQV} with minor modifications.

\begin{lemma}\label{lem: comparison}
For $i = 0,1$, assume $\rho_0^i\in L^1(\BR^d)\cap BV(\RR^d)$ such that $\rho_0^i\in [0,1]$ almost everywhere.
Take a finite $T>0$ and denote $Q_T = \BR^d\times [0,T]$.
Let $f^i\in L^\infty([0,T];L^1(\BR^d))\cap L^\infty(Q_T)$.
Let $(\rho^i, p^i)$ $(i = 0,1)$ be weak solutions of
\[
\rho^i_t - \nabla\cdot(\rho^i \nabla p^i) = f^i,\quad \rho^i\leq 1,\quad p^i\in P_\infty (\rho^i),\quad \rho^i(x,0) = \rho^i_0,
\]
respectively.
If $\rho_0^0\leq \rho_0^1$ almost everywhere in $\BR^d$ and $f^0\leq f^1$ almost everywhere in $Q_T$, then $\rho^0\leq\rho^1$ almost everywhere in $Q_T$.
\end{lemma}

\begin{remark}\label{rmk: compact support}
In a similar spirit, one can show that if $(\rho, p,n)$ is a weak solution of \eqref{eqn: P} on $Q_T$ with $\rho_0$ being compactly supported, then $\rho$ and $p$ are compactly supported in $Q_T$.
Indeed, assuming $\rho_0\leq \chi_{B_{r_0}}$ for some $r_0>0$, we can prove $\rho\leq \tilde{\rho}$ and $p\leq \tilde{p}$, where
\[
\tilde{\rho}(x,t) := \chi_{B_{r(t)}}(x),\;
\tilde{p}(x,t) := \f{\|n_0\|_{L^\infty(\BR^d)} }{2d} \left(r(t)^2-|x|^2\right)_+,\mbox{ and }
r(t):= r_0\exp\left(\f{1}{d}\|n_0\|_{L^\infty(\BR^d)} t\right).
\]
Note that $(\tilde{\rho},\tilde{p})$ is chosen to be a weak solution (in the sense of Definition \ref{def: weak solution p equation}) of
\[
\partial_t \tilde{\rho} -\nabla (\tilde{\rho}\nabla \tilde{p}) = \|n_0\|_{L^\infty(\BR^d)}\tilde{\rho},\quad \tilde{\rho}\leq 1,\quad \tilde{p}\in P_\infty(\tilde{\rho}),\quad \tilde{\rho}(x,0) = \chi_{B_{r_0}}(x).
\]
We skip the details.
\end{remark}

Now we present uniqueness of the weak solution for \eqref{eqn: P}.

\begin{prop}\label{prop: uniqueness}
Under the assumption of Proposition \ref{prop:existence},
the weak solution of \eqref{eqn: P} is unique.
\begin{proof}
Fix $T>0$.  Suppose $(\rho,p,n)$ and $(\tilde{\rho}, \tilde{p}, \tilde{n})$ are two weak solutions.
We first show an $L^1$-contraction principle regarding $\rho$ and $\tilde{\rho}$.

Denote
\[
f = (n-b)\rho,\quad \tilde{f}= (\tilde{n}-b)\tilde{\rho},\quad f_* = \max\big\{f,\tilde{f}\big\}.
\]
Since $f,\tilde{f} \in L^\infty([0,T];BV(\BR^d))$ (c.f.\;Proposition \ref{prop:existence}), we have $f_* \in L^\infty([0,T];BV(\BR^d))$.
Let $(\rho_*,p_*)$ be a weak solution (see Remark \ref{rmk: existence of solution of model problem}) of
\[
(\rho_*)_t - \nabla\cdot(\rho_* \nabla p_*) = f_*,\quad \rho_*\leq 1,\quad p_*\in P_\infty (\rho_*),\quad \rho_*(x,0) = \rho_0.
\]
Then by Lemma \ref{lem: comparison}, $\rho,\tilde{\rho}\leq \rho_*$ almost everywhere in $Q_T$.

Take an arbitrary $t_*\in (0,T]$. Let $\zeta:\RR^d\to\RR$ be a smooth non-negative function such that $\zeta=1$ on the unit ball, and $\zeta=0$ outside the ball of radius 2.  For any $R>0$, set $\zeta_R(x):=\zeta(x/R)$.
With $\delta <t_*$, let $\eta_\delta(t)\in C^\infty([0,T])$, such that it is non-increasing, $\eta_\delta \equiv 1$ on $[0,t_*-\delta]$, and $\eta_\delta\equiv 0$ on $[t_*,T]$.
Then we take $\psi(x,t) = \zeta_R(x)\eta_\delta(t)\in H^1(Q_T)$ in Definition \ref{def: weak solution} and Definition \ref{def: weak solution p equation}, and take a difference of them to obtain that
\[
\int_0^T -\partial_t \eta_\delta(t) \big\|\big(\rho(\cdot,t)-\rho_*(\cdot,t)\big)\zeta_R \big\|_{L^1(\BR^d)} \,dt
=
\int_0^T \eta_\delta(t) \big\|\big(f(\cdot,t)-f_*(\cdot,t)\big)\zeta_R \big\|_{L^1(\BR^d)}\,dt +\epsilon_R,
\]
where $\epsilon_R$ is an error term with $|\epsilon_R|\lesssim R^{-2}(\norm{p}_{L^1(Q_T)}+\norm{p_*}_{L^1(Q_T)})$ and we used the facts that $\rho\leq \rho_*$ and $f\leq f_*$.
Here one can derive an estimate for $\norm{p_*}_{L^1(Q_T)}$ as in Lemma \ref{lem:scheme_estimates}.
Sending $R\to\infty$ and then $\delta \to 0$, we can use the time-continuity of $\rho$ and $\rho_*$ (see Definition \ref{def: weak solution} and Definition \ref{def: weak solution p equation}), to obtain
\[
\|\rho_*(\cdot,t_*)-\rho(\cdot,t_*)\|_{L^1(\BR^d)}
=\|f_* - f\|_{L^1(\BR^d\times [0,t_*])}.
\]
Similarly,
\[
\|\rho_*(\cdot,t_*)-\tilde{\rho}(\cdot,t_*)\|_{L^1(\BR^d)}
=\|f_* - \tilde{f}\|_{L^1(\BR^d\times [0,t_*])}.
\]
As a result,
\beq
\begin{split}
\|\tilde{\rho}(\cdot,t_*)-\rho(\cdot,t_*)\|_{L^1(\BR^d)}
\leq &\;
\|f_* - f\|_{L^1(\BR^d\times [0,t_*])}+\|f_* - \tilde{f}\|_{L^1(\BR^d\times [0,t_*])} \\
= &\; \|f - \tilde{f}\|_{L^1(\BR^d\times [0,t_*])},
\end{split}
\label{eqn: L^1 contraction parabolic}
\eeq
where in the last equality, we used the definition of $f_*$.

Now by the definition of $f$ and $\tilde{f}$, for any $t\in [0,T]$,
\[
\begin{split}
&\;\|\tilde{\rho}(\cdot,t)-\rho(\cdot,t)\|_{L^1(\BR^d)}\\
\leq &\;
\int_0^t (\|n_0\|_{L^\infty}+b)\|\tilde{\rho}(\cdot,\tau) - \rho(\cdot,\tau)\|_{L^1(\BR^d)}
+ \|\tilde{n}(\cdot,\tau) - n(\cdot,\tau)\|_{L^1(\BR^d)}
\,d\tau.
\end{split}
\]
Using the Duhamel's formula for the heat equation, for every $t\in [0,T]$ and almost every $x\in \RR^d$, we have
\[
\tilde{n}(\cdot,t) - n(\cdot, t)= - \int_0^{t} e^{(t-\tau)D\Delta} \big(\tilde{n}(\cdot,\tau)\tilde{\rho}(\cdot, \tau)-n(\cdot,\tau)\rho(\cdot, \tau)\big)\,d\tau.
\]
Note that this formula is still valid even when $D=0$.   Since the heat kernel is a contraction on $L^1$, it follows that
\[
\begin{split}
\norm{\tilde{n}(\cdot,t)-n(\cdot,t)}_{L^1(\RR^d)}\leq &\; \int_0^t \norm{\tilde{n}(\cdot,\tau)\tilde{\rho}(\cdot, \tau)-n(\cdot,\tau)\rho(\cdot, \tau)}_{L^1(\RR^d)}\, d\tau\\
\leq &\; \int_0^t \norm{n_0}_{L^{\infty}(\RR^d)}\norm{\tilde{\rho}(\cdot,\tau)-\rho(\cdot,\tau)}_{L^1(\RR^d)}+\norm{\tilde{n}(\cdot,\tau)-n(\cdot,\tau)}_{L^1(\RR^d)}\, d\tau.
\end{split}
\]
Thus,
\[
\begin{split}
&\;\|\tilde{\rho}(\cdot,t)-\rho(\cdot,t)\|_{L^1(\BR^d)}+\norm{\tilde{n}(\cdot,t)-n(\cdot,t)}_{L^1(\RR^d)}\\
\leq &\;
\int_0^t (2\|n_0\|_{L^\infty}+b)\|\tilde{\rho}(\cdot,\tau) - \rho(\cdot,\tau)\|_{L^1(\BR^d)}
+ 2\|\tilde{n}(\cdot,\tau) - n(\cdot,\tau)\|_{L^1(\BR^d)}
\,d\tau.
\end{split}
\]
It now follows from Gronwall's inequality that $\|\tilde{\rho}(\cdot,t)-\rho(\cdot,t)\|_{L^1(\BR^d)}+\norm{\tilde{n}(\cdot,t)-n(\cdot,t)}_{L^1(\RR^d)}=0$ for all $t\in [0,T]$, and thus $\rho(\cdot,t) = \tilde{\rho}(\cdot,t)$ and $n(\cdot,t) =\tilde{n}(\cdot,t)$ for all $t\in [0,T]$.
Lastly, thanks to the weak formulation \eqref{eqn: weak formulation}, for any $\psi\in L^2([0,T];H^1(\RR^d))$
\beqo
\int_0^T\int_{\RR^d}\nabla \psi\cdot \nabla \big(p-\tilde{p}\big) \,dx \,dt = 0.
\eeqo
Choosing $\psi$ that approximates $(p-\tilde{p})$, we conclude that $p=\tilde{p}$.
\end{proof}
\end{prop}

\section{The Case of Zero Nutrient Diffusion and No Death}
\label{sec: D is 0}

In the rest of the paper, we will always assume $b = D = 0$ in \eqref{eqn: P}.
This section aims at studying qualitative properties of the solution.
Recall that in this case, $\eta:= n_0-n$ solves
\beq
\partial_t \eta = (n_0-\eta)\rho,\quad \eta(x,0) = 0.
\label{eqn: eta equation}
\eeq
Given $(\rho,p,n)$ as a weak solution of \eqref{eqn: P}, it is straightforward to show that $\eta$ is continuous in $L^1(\BR^d)$ with respect to time.

The assumption $b=0$ implies the tumor region is always expanding.
This enables us to study the $\rho$-equation of \eqref{eqn: P} through the so-called Baiocchi transform \cite{baiocchi}.
Indeed, if we define $w(x,t)=\int_0^t p(x,s)\, ds$, we see that for any $t\geq 0$,
\[
w(x,t)(1-\rho(x,t))=\int_0^t p(x,s)(1-\rho(x,t))\, ds\leq \int_0^t p(x,s)(1-\rho(x,s))\, ds=0
\]
everywhere on $\BR^d$.
Hence, by integrating the $\rho$-equation of \eqref{eqn: P} over the time interval $[0,t]$ and using the time continuity of $\rho$ and $\eta$ in $L^1(\BR^d)$, we get the elliptic equation
\begin{equation}\label{eq:elliptic}
\rho-\Delta w=\rho_0+\eta, \quad w(1-\rho) = 0,\quad \rho\in [0,1],\quad w\geq 0
\end{equation}
along with the boundary condition that $w$ vanishes at infinity.

The main advantage of this formulation is that $w$ enjoys much better regularity as compared to $p$ (clearly $w\in W^{2,r}(\RR^d)$ for any $r<\infty$) while still describing the active tumor region.
Section \ref{sec: elliptic formulation} will be focused on proving comparison/contraction properties of this elliptic equation.

On the other hand, that $D = 0$ gives rise to many nice properties of the model.
For example, it gives a direct pointwise dependence of the nutrient variables $n$ and $\eta$ on the density $\rho$
\beq
\eta(x,t)=n_0(x)\left[1-\exp\left(-\int_0^t \rho(x,s)\, ds\right)\right].
\label{eqn: eta formula original}
\eeq
This allows us to extend the comparison/contraction results for $\rho$-equation \eqref{eq:elliptic} to the full PDE \eqref{eqn: P}.
See Section \ref{sec: bounded nutrient}.

\subsection{The elliptic formulation}
\label{sec: elliptic formulation}
Let us start from the elliptic equation \eqref{eq:elliptic}.
We first establish the following comparison principles and stability estimates.

\begin{prop}\label{prop:comparison}
Let $\Omega\subset \RR^d$ be an open set with a smooth (possibly empty) boundary.
Suppose that for $i\in \{0,1\}$ we have non-negative functions $\rho^i, f^i\in C([0,T];L^1(\Omega))\cap L^{\infty}(\Omega\times[0,T])$ and $w^i\in L^{\infty}([0,T]; H^2(\Omega))$ such that for all $t\in [0,T]$
\[
\rho^i-\Delta w^i=f^i, \quad w^i(1-\rho^i)=0,\quad \rho^i \in[0,1],\quad w^i\geq 0
\]
on $\Omega$.
If $w^0(x,t)\leq w^1(x,t)$ for all $(x,t)\in \partial\Omega\times [0,T]$,  and $f^0(x,t)\leq f^1(x,t)$ for every $t\in [0,T]$ and almost every $x\in\Omega$, then for all $t\in [0,T]$ and almost every $x\in \Omega$ we have $\rho^0(x,t)\leq \rho^1(x,t)$ and $w^0(x,t)\leq w^1(x,t)$.

\end{prop}
\begin{remark}
Let us emphasize that this comparison property does not immediately imply that solutions to \eqref{eqn: P} satisfy a comparison principle, since we have not shown that the nutrient variables stay ordered along solutions to the PDE (however, see Theorem \ref{thm:contraction}).
\end{remark}
\begin{proof}
Taking the difference of the two equations and integrating against $(w^0-w^1)_+$ on $\Omega\times [0,T]$, we obtain
\[
\int_0^T \int_{\Omega}(\rho^0-\rho^1)(w^0-w^1)_++|\nabla (w^0-w^1)_+|^2\,dx\,dt
=
\int_0^T \int_{\Omega} (f^0-f^1)(w^0-w^1)_+ \,dx\,dt,
\]
where we have used the fact that $(w_0-w_1)_+$ vanishes on $\partial\Omega\times [0,T]$.
The right hand side of the equation is clearly non-positive.
On the other hand, if $w^0>w^1$ then we must have $w^0>0$. Therefore $(\rho^0-\rho^1)(w^0-w^1)_+=(1-\rho^1)(w^0-w^1)_+ \geq 0.$ Thus, the left hand side of the equation is non-negative.  Hence, both sides must be equal to zero, which allows us to conclude that $(w^0-w^1)_+=c(t)$ in $\Omega\times [0,T]$ for some non-negative constant $c$ that depends only on time.
Since both $w^0$ and $w^1$ approach zero at infinity and $w^0\leq w^1$ on $\partial\Omega\times [0,T]$, it follows that $c=0$.
Thus, $w^0\leq w^1$ almost everywhere on $\Omega\times [0,T]$.

It is clear from the equation and the fact $w^0\leq w^1$ a.e.\;that $\{(x,t)\in \Omega \times [0,T]: \rho^0(x,t)>\rho^1(x,t)\}\subset \{(x,t)\in \Omega\times [0,T]: w^1(x,t)=0, w^0(x,t)=0\}$ up to a set of measure zero.
Since the equation implies that $w^i\in L^{\infty}([0,T];W^{2,r}(\RR^d))$ for any $r<\infty$, it follows that $\Delta w^i = 0$ a.e.\;on the set where $w^i=0$.
Thus, $\rho^i=f^i$ almost everywhere on $\{(x,t)\in \Omega\times [0,T]: w^1(x,t)=0, w^0(x,t)=0\}$.  This allows us to conclude that $\rho^0\leq \rho^1$ almost everywhere on $\Omega\times [0,T]$.
\end{proof}

A similar argument gives us the following $L^1$-stability property.
It is a strengthening of Lemma \ref{lem: comparison} and the $L^1$-contraction principle \eqref{eqn: L^1 contraction parabolic} in the proof of Proposition \ref{prop: uniqueness}.

\begin{prop}\label{prop:contraction}
Let $\Omega\subset \RR^d$ be an open set with a smooth (possibly empty) boundary. Suppose that for $i\in \{0,1\}$, $\rho_0^i(x)\in [0,1]$ are integrable on $\Omega$,
$\rho^i, \eta^i\in C([0,T];L^1(\Omega))\cap L^{\infty}(\Omega\times[0,T])$ and $w^i\in L^{\infty}([0,T]; H^2(\Omega))$, such that
\[
\rho^i-\Delta w^i=\rho_0^i+\eta^i, \quad w^i(1-\rho^i)=0,\quad \rho^i\in [0,1],\quad w^i\geq 0
\]
on $\Omega$.  If $w^0(x,t)=w^1(x,t)$ for all $(x,t)\in \partial\Omega\times [0,T]$, then for all $t\in [0,T]$
\[
\big\|(\rho^1(\cdot,t)-\rho^0(\cdot,t))_+\big\|_{L^1(\Omega)}\leq \left\|\Big(\rho^1_0(\cdot)-\rho^0_0(\cdot)+\eta^1(\cdot,t)-\eta^0(\cdot,t)\Big)_+\right\|_{L^1(\Omega)}.
\]
\end{prop}
\begin{proof}
Given small $\delta>0$, let $f_{\delta}:\RR\to\RR$ be a smooth increasing function such that $f_{\delta}(a)=0$ for all $a\leq 0$ and $f_{\delta}(a)=1$ for all $a\geq \delta$.  Choose some $0\leq t_1<t_2\leq T$.  Taking the difference of the two equations and integrating against $f_{\delta}(w^1-w^0)$ on $\Omega\times [t_1,t_2]$, we have
\[
\begin{split}
& \int_{\Omega\times [t_1,t_2]} (\rho^1-\rho^0)f_{\delta}(w^1-w^0)+f_{\delta}'(w^1-w^0)|\nabla (w^1-w^0)|^2\,dx \,dt \\
& = \int_{\Omega\times [t_1,t_2]} (\rho^1_0-\rho^0_0+\eta^1-\eta^0)f_{\delta}(w^1-w^0)\,dx \,dt .
\end{split}
\]
Sending $\delta\to 0$, we see that
\[
\int_{\Omega\times [t_1,t_2]} (\rho^1-\rho^0)\sgn_+(w^1-w^0)\,dx \,dt \leq \int_{\Omega\times [t_1,t_2]} (\rho^1_0-\rho^0_0+\eta^1-\eta^0)\sgn_+(w^1-w^0)\,dx \,dt .
\]
Let $E=\{(x,t)\in \Omega\times [t_1,t_2]: w^1(x,t)\leq w^0(x,t), \, \rho^0(x,t)<\rho^1(x,t)\}$.
Then $w^1, w^0=0$ almost everywhere on $E$.
Thus the regularity of $w^i$ implies that $\rho^i=\rho^i_0+\eta^i$ almost everywhere on $E$.
Now we use the above inequality to derive that
\[
\begin{split}
&\;\int_{\Omega\times [t_1,t_2]} (\rho^1-\rho^0)_+ \,dx \,dt \\
= &\;
\int_{\Omega\times [t_1,t_2]} (\rho^1-\rho^0)\sgn_+(w^1-w^0)\,dx \,dt
+\int_E(\rho^1-\rho^0)_+ \,dx \,dt \\
\leq &\; \int_{\Omega\times [t_1,t_2]} (\rho^1_0-\rho^0_0+\eta^1-\eta^0)\sgn_+(w^1-w^0)\,dx \,dt
+\int_E (\rho^1_0-\rho^0_0+\eta^1-\eta^0)_+ \,dx \,dt \\
\leq &\; \big\|(\rho^1_0-\rho^0_0+\eta^1-\eta^0)_+\big\|_{L^1(\Omega\times [t_1,t_2])}.
\end{split}
\]
Then the result follows since $t_1$ and $t_2$ are arbitrary and $\rho^i$ and $\eta^i$ are continuous in $L^1(\Omega)$ with respect to time.
\end{proof}

Thanks to the symmetries of the Laplacian, the above stability result implies the following derivative estimates.

\begin{prop}\label{prop:fbv}
Suppose $(\rho, \rho_0, w, \eta)$ solve
\[
\rho-\Delta w=\rho_0+\eta, \quad w(1-\rho)=0,\quad \rho\in[0,1],\quad w\geq 0
\]
on $\RR^d\times [0,T]$, where $\rho_0\in [0,1]$ is integrable on $\BR^d$, $\eta\geq 0$,
$\rho, \eta\in C([0,T];L^1(\BR^d))\cap L^{\infty}(\BR^d\times[0,T])$ and $w\in L^{\infty}([0,T]; H^2(\BR^d))$.
Then for any antisymmetric matrix $A\in \RR^{d\times d}$, vector $v\in \RR^d$, and scalar $\lambda\geq 0$, we have for all $t\in [0,T]$,
\[
\norm{\nabla \rho(x, t)\cdot (Ax+\lambda x+v)}_{L^1(\RR^d)}\leq \norm{\nabla (\rho_0(x, t)+\eta(x, t))\cdot (Ax+\lambda x+v)}_{L^1(\RR^d)}.
\]
In particular, for all $t\in [0,T]$,
\[
\norm{\rho(x, t)}_{BV(\RR^d)}\leq \norm{\rho_0(x, t)+\eta(x, t)}_{BV(\RR^d)}.
\]
\end{prop}
\begin{proof}
 Given an antisymmetric matrix $A$ and $\lambda\geq 0$ we define $A_s:=\exp((\lambda I+A)s)$ for each $s\in [0,1]$.  A direct computation reveals that
\[
\partial_s (A_sA_s^\intercal)=A_s(\lambda I+A)A_s^\intercal+A_s(\lambda I+A)^\intercal A_s^\intercal=2\lambda A_sA_s^\intercal.
\]
Since $A_0A_0^\intercal=I$, 
we can integrate directly to see that $A_sA_s^\intercal=e^{2\lambda s}I$.   Given $v\in \RR^d$,  if we define $T_s:\RR^d\to\RR^d$ such that $T_s(x)=A_sx+sv$, then $T_s$ is a smooth curve of affine transformations such that $\partial_s T_0(x)=\lambda x+Ax+v$ and $DT_sDT_s^\intercal =e^{2\lambda s}I$ for all $s\in [0,1]$.

Define $\rho_s:=\rho\circ T_s$, $w_s:=e^{-2\lambda s}w\circ T_s$, $\eta_s:=\eta\circ T_s$, and $\rho_{0,s}:=\rho_0\circ T_s$.
We can then compute
\[
\Delta (w_s)=DTDT^\intercal:(D^2w\circ T_s)+\Delta T_s\cdot \nabla w_s=(\Delta w)\circ T_s,
\]
where we have used the fact that $DTDT^\intercal:(D^2w\circ T_s)=e^{2\lambda s}\tr(e^{-2\lambda s}D^2w\circ T_s)$ and $\Delta T_s=0$.
Thus,
\[
\rho_s-\Delta w_s=\big(\rho-\Delta w\big)\circ T_s=(\rho_0+\eta)\circ T_s=\rho_{0,s}+\eta_s.
\]
It is also clear that $w_s(1-\rho_s)=0$.
Proposition \ref{prop:contraction} now implies that for all $t\in [0,T]$ and every $s\in [0,1]$,
\[
\norm{\rho(\cdot, t)-\rho_s(\cdot, t)}_{L^1(\RR^d)}\leq \norm{\rho_0(\cdot,t)-\rho_{0,s}(\cdot, t)+\eta(\cdot, t)-\eta_s(\cdot, t)}_{L^1(\RR^d)}.
\]
Dividing both sides by $s$ and then sending $s\to 0$, the first result follows. The second result follows from choosing $A=0$ and $\lambda=0$ and then taking supremum over the vectors $v$ with $|v|\leq 1$.
\end{proof}

\subsection{The full system with bounded initial nutrient}

\label{sec: bounded nutrient}

Now we study the full system \eqref{eqn: P} under the assumption that the initial nutrient $n_0$ is bounded.
In what follows, we shall always assume the initial density $\rho_0$ is a \emph{patch}, namely,
\[
\rho_0\in L^1(\BR^d)\cap BV(\BR^d),\mbox{ and }\rho_0(x)\in \{0,1\}\mbox{ almost everywhere.}
\]
While we already know that this must imply that $\rho(x,t)\in \{0,1\}$ almost everywhere for all $t$, the following lemma shows that in this case $\eta$ will be concentrated on the support of $\rho$.



\begin{lemma}\label{lem:01}
Suppose $\rho_0(x)\in \{0,1\}$ almost everywhere in $\RR^d$.
Then for all $t\in [0,T]$, $\rho\in \{0,1\}$ and $(1-\rho)\eta=0$ almost everywhere in $\RR^d$.
\end{lemma}
\begin{proof}
Since $\eta\geq 0$ and $\partial_t \rho \geq 0$ (c.f.\;Proposition \ref{prop:contraction}), a direct computation shows that
\[
\frac{d}{dt} \int_{\RR^d} (1-\rho(x,t))\eta(x,t)\, dx\leq \int_{\RR^d} (1-\rho(x,t))\rho(x,t)(n_0(x)-\eta(x,t))\,dx=0.
\]
Thus, Gronwall's inequality and the non-negativity of $(1-\rho)\eta$ implies that, for all $t\in [0,T]$, $(1-\rho)\eta=0$ almost everywhere in $\RR^d$.
%
\end{proof}

The next result gives a more complete characterization of the tumor patch, i.e., the set $\{\rho=1\}$. Formally it says that the tumor patch coincides with the support of the pressure variable, characterizing the evolution of the tumor as a free boundary problem driven by the pressure.  It will be useful later in Section \ref{sec:regularity}. When the tumor boundary is regular, this is
easy to prove since the pressure solves the elliptic equation $-\Delta p = n$ in the interior of the tumor. However such regularity is unknown a priori, so we instead argue with the more regular variable $w$, the time integral of $p$.

\begin{lemma}\label{lem:rho_1_characterization}
Assume $\rho_0(x)\in \{0,1\}$ almost everywhere.
Suppose $(\rho, \rho_0, w, \eta)$ satifies
\beq
\rho-\Delta w=\rho_0+\eta, \quad w(1-\rho)=0, \quad \rho\in \{0,1\},\quad w\geq 0,\quad \partial_t \eta=(n_0-\eta)\rho.
\label{eqn: w equation}
\eeq
Then for all $t\in (0,T]$,
\[
\{x\in \RR^d: \rho(x,t)=1\}=\{x\in \RR^d: w(x,t)>0\}\cup\{x\in\RR^d: \rho_0(x)=1, n_0(x)=0\},
\]
up to a measure zero set in $\BR^d$.
\end{lemma}

\begin{proof}

Thanks to the regularity of $w$,  $\Delta w$ vanishes almost everywhere on the set where $w$ vanishes.
Thus, combining the elliptic equation and the relation $w(1-\rho)=0$, it follows that for every $t\in (0,T]$
\begin{equation}\label{eq:rho_decomposition}
\rho(x,t)=\chi_{w}(x,t)+(1-\chi_w(x,t))(\rho_0(x)+\eta(x,t))
\end{equation}
almost everywhere,
where $\chi_w$ is the characteristic function of the set $\{w>0\}$.

Using Lemma \ref{lem:01} we see that $\eta$ solves $\partial_t \eta=n_0\rho-\eta$.  Therefore, we have the following integral representation for $\eta$:
\[
\begin{split}
\eta(x,t)= &\; n_0(x)\int_0^t \rho(x,s)e^{s-t}\, ds\\
= &\; n_0(x)\int_0^t \Big(\chi_w(x,s)+(1-\chi_w(x,s))\big(\rho_0(x)+\eta(x,s)\big)\Big)e^{s-t}\, ds.
\end{split}
\]
Define
\[
h(x,t):=\eta(x,t)(1-\chi_w(x,t)).
\]
Since $\chi_w$ is increasing in time, it follows that $\chi_w(x,s)(1-\chi_w(x,t))=0$ and $(1-\chi_w(x,s))(1-\chi_w(x,t))=(1-\chi_w(x,s))$ whenever $s\leq t$.  Therefore, multiplying the integral representation of $\eta$ by $(1-\chi_w(x,t))$, we obtain that
\[
\begin{split}
h(x,t)= &\; n_0(x)\int_0^t (1-\chi_w(x,s))\big(\rho_0(x)+\eta(x,s)\big) e^{s-t}\, ds\\
=&\; n_0(x)\int_0^t \big((1-\chi_w(x,s))\rho_0(x)+h(x,s)\big)e^{s-t}\, ds
\end{split}
\]
holds almost everywhere.
For almost every $x$ such that $n_0(x)\rho_0(x)=0$, Gronwall's inequality implies that $h(x,t)=0$ for all $t\in [0,T]$.

By \eqref{eq:rho_decomposition} and the definition of $h$, we may write $\{(x,t): \rho(x,t)=1\}=\{(x,t): w(x,t)>0\}\cup\{(x,t): \rho_0(x)+h(x,t)=1\}$ up to a measure-zero set.   From our work above, when $n_0(x)\rho_0(x)=0$, $h(x,t)=0$ almost everywhere.
When $n_0(x)\rho_0(x)\neq 0$ we have $\rho_0(x)=1$ and $n_0(x)>0$, which implies that $\rho_0(x)+\eta(x,t)>1$ for all $t>0$ and it is non-decreasing.
For almost every such $x$, we must have $w(x,t)>0$ for all $t>0$, as otherwise \eqref{eqn: w equation} would not hold.

This completes the proof.
\end{proof}

Next, we state two results on the stability of the full system \eqref{eqn: P}.
They respectively extend Proposition \ref{prop:contraction} and Proposition \ref{prop:fbv}.

\begin{theorem}\label{thm:contraction}
Suppose that $(\rho^0, p^0, \eta^0)$ and $(\rho^1, p^1, \eta^1)$ are solutions to the PDE \eqref{eqn: P} starting from initial datum $(\rho_0^0, n_0^0)$ and $(\rho_0^1, n_0^1)$, respectively. Fix a domain $D\subset \RR^d$.
Suppose $\rho_0^0$ and $\rho_0^1$ are patches, $n^0_0\in L^{\infty}(D)$, and $p^0=p^1$ on $\partial D\times [0,T]$.
Denote
\[
N_D(t)=\int_0^t e^{(\norm{n^0_0}_{L^{\infty}(D)}-1)\tau}\,d\tau
=
\begin{cases}
\frac{e^{(\norm{n^0_0}_{L^{\infty}(D)}-1)t}-1}{\norm{n^0_0}_{L^{\infty}(D)}-1} & \textup{if} \;\norm{n^0_0}_{L^{\infty}(D)}\neq 1,\\
t &\textup{otherwise},
\end{cases}
\]
and
\[
M_D(t)= \norm{n^0_0}_{L^{\infty}(D)}N_D(t)+1 =
\begin{cases}
\frac{\norm{n^0_0}_{L^{\infty}(D)}e^{(\norm{n^0_0}_{L^{\infty}(D)}-1)t}-1}{\norm{n^0_0}_{L^{\infty}(D)}-1} & \textup{if} \;\norm{n^0_0}_{L^{\infty}(D)}\neq 1,\\
t+1 &\textup{otherwise}.
\end{cases}
\]
Then for all $t\in [0,T]$,
\begin{equation*}
\norm{(\rho^1-\rho^0)_+}_{L^1(D\times \{t\})}
\leq
N_D(t)\norm{(n^1_0-n^0_0)_+}_{L^1(D)}
+ M_D(t)\norm{(\rho_0^1-\rho_0^0)_+}_{L^1(D)}.
\end{equation*}
If in addition, $\rho_0^0\leq \rho_0^1$ and $n_0^0\leq n_0^1$ for almost every $x\in D$, then for all $t\in [0,T]$ we have $\rho^0(x,t)\leq \rho^1(x,t)$ for almost every $x\in D$.
\end{theorem}
\begin{remark}
Compared with the proof of Proposition \ref{prop: uniqueness}, this theorem provides an improved estimate by taking advantage of the conditions $b = 0$ and $\rho^i$ $(i = 0,1)$ are patch solutions.
\end{remark}
\begin{proof}
Thanks to Proposition \ref{prop:contraction}, for every $t\in [0,T]$ we have the contraction inequality
\[
\norm{(\rho^1-\rho^0)_+}_{L^1(D\times \{t\})}\leq \norm{(\rho^1_0-\rho_0^0)_+}_{L^1(D)}+\norm{(\eta^1-\eta^0)_+}_{L^1(D\times \{t\})}.
\]
On the other hand, by direct computation and Lemma \ref{lem:01},
\[
\begin{split}
&\;\frac{d}{dt} \norm{(\eta^1-\eta^0)_+}_{L^1(D\times \{t\})}+\norm{(\eta^1-\eta^0)_+}_{L^1(D\times \{t\})}\\
\leq &\; \norm{(n^1_0\rho^1-n^0_0\rho^0)_+}_{L^1(D\times \{t\})}\\
\leq &\; \norm{(n_0^1-n^0_0)_+}_{L^1(D)}+\norm{n^0_0}_{L^{\infty}(D)}\norm{(\rho^1-\rho^0)_+}_{L^1(D\times \{t\})}.
\end{split}
\]
Plugging in the contraction inequality, we see that
\[
\begin{split}
\frac{d}{dt} \norm{(\eta^1-\eta^0)_+}_{L^1(D\times \{t\})}
\leq &\;
\big(\norm{n^0_0}_{L^{\infty}(\RR^d)}-1\big) \norm{(\eta^1-\eta^0)_+}_{L^1(D\times \{t\})}\\
&\;+\norm{(n_0^1-n^0_0)_+}_{L^1(D)}+\norm{n^0_0}_{L^{\infty}(D)}\norm{(\rho^1_0-\rho_0^0)_+}_{L^1(D)}.
\end{split}
\]
Then Gronwall's inequality implies that
\[
\norm{(\eta^1-\eta^0)_+}_{L^1(D\times \{t\})}\leq N_D(t)\Big(\norm{(n^1_0-n^0_0)_+}_{L^1(D)}+\norm{n^0_0}_{L^{\infty}(\RR^d)}\norm{(\rho^1_0-\rho_0^0)_+}_{L^1(D)}\Big),
\]
and thus
\[
\norm{(\rho^1-\rho^0)_+}_{L^1(D\times \{t\})}
\leq N_D(t)\norm{(n_0^1-n^0_0)_+}_{L^1(D)}
+ M_D(t)\norm{(\rho^1_0-\rho_0^0)_+}_{L^1(D)}.
\]

The final comparison property is automatic from the above estimates.
\end{proof}

\begin{prop}
\label{prop: BV estimates}
Let $(\rho, p, \eta)$ be a solution to \eqref{eqn: P} starting from initial data $\rho_0$ and $n_0$.  Given an antisymmetric matrix $A\in \RR^{d\times d}$ and some $g\in L^1(\RR^d)$ define
\[
\norm{g}_{BV_A(\RR^d)}=\sup_{\vp\in C^{\infty}_c(\RR^d), \norm{\vp}_{L^{\infty}(\RR^d)}\leq 1}\int_{\RR^d} g(x) Ax\cdot \nabla \vp\,dx.
\]
If $\rho_0\in BV(\RR^d)$ is a patch and $n_0\in L^{\infty}(\RR^d)\cap BV(\RR^d)$, then for all $t\in [0,T]$,

\[
\norm{\rho}_{BV(\RR^d\times \{t\})}\leq N(t) \norm{n_0}_{BV(\RR^d)}+M(t)\norm{\rho_0}_{BV(\RR^d)},
\]
and
\[
\norm{\rho}_{BV_A(\RR^d\times \{t\})}
\leq
N(t) \norm{n_0}_{BV_A(\RR^d)} + M(t)\norm{\rho_0}_{BV_A(\RR^d)}.
\]
Here $N(t):= N_{\BR^d}(t)$ and $M(t):= M_{\BR^d}(t)$ are defined as in Theorem \ref{thm:contraction}.

\end{prop}
\begin{proof}
Proposition \ref{prop:fbv} immediately gives us the bound
\[
\norm{\rho(\cdot,t)}_{BV(\RR^d)}\leq \norm{\rho_0}_{BV(\RR^d)}+\norm{\eta(\cdot,t)}_{BV(\RR^d)}.
\]
Using \eqref{eqn: eta formula original}, it is clear that $\rho, \eta\in L^{\infty}([0,T];BV(\RR^d))$.
To obtain a better bound, we differentiate the $\eta$-equation
\[
\partial_t \nabla \eta=\rho\nabla n_0+n_0\nabla \rho-\nabla \eta,
\]
which allows us to estimate
\[
\begin{split}
\frac{d}{dt} \big(e^t \norm{\eta(\cdot,t)}_{BV(\RR^d)}\big)\leq &\;
e^t\left(\norm{n_0}_{BV(\RR^d)}
+\norm{n_0}_{L^{\infty}(\RR^d)}\norm{\rho(\cdot,t)}_{BV(\RR^d)}\right)\\
\leq &\; e^t \left(\norm{n_0}_{BV(\RR^d)}
+\norm{n_0}_{L^{\infty}(\RR^d)}\norm{\rho_0}_{BV(\RR^d)}
+\norm{n_0}_{L^{\infty}(\RR^d)} \norm{\eta(\cdot,t)}_{BV(\RR^d)}\right).
\end{split}
\]
Rearranging this yields
\[
\frac{d}{dt}\big(e^{(1-\norm{n_0}_{L^{\infty}(\RR^d)})t} \norm{\eta(\cdot,t)}_{BV(\RR^d)}\big)
\leq
e^{(1-\norm{n_0}_{L^{\infty}(\RR^d)})t} \Big(\norm{n_0}_{BV(\RR^d)}+\norm{n_0}_{L^{\infty}(\RR^d)}\norm{\rho_0}_{BV(\RR^d)}\Big).
\]
Thus,
\[
\norm{\eta(\cdot,t)}_{BV(\RR^d)} \leq N(t)\Big( \norm{n_0}_{BV(\RR^d)}+\norm{n_0}_{L^{\infty}(\RR^d)}\norm{\rho_0}_{BV(\RR^d)}\Big).
\]
The $BV$-estimate now follows.
The estimate for the $BV_A$-norm can be proved analogously.
\end{proof}

%
%
%



By the comparison principle Theorem \ref{thm:contraction}, one can see that any solution starting with an initial nutrient such that $\norm{n_0}_{L^{\infty}(\RR^d)}<1$ must have bounded mass for all time.
Indeed, we may take in Theorem \ref{thm:contraction} that $\rho_0^0 = 0$, $\rho_0^1 = \rho_0$, and $n_0^0 = n_0^1 = n_0$.
In such case, the solution approaches a stationary state as $t\to\infty$.
The following theorem characterizes the stationary state and provides a convergence rate.

\begin{theorem}
\label{thm: convergence when n_0 is less than 1}
Suppose $\norm{n_0}_{L^{\infty}(\RR^d)}<1$ and $\rho_0\in \{0,1\}$ and that $(\rho, p,\eta)$ is a solution to the PDE \eqref{eqn: P} starting from the initial data $(\rho_0, n_0)$.
Let $\rho_{\infty}$ solve the elliptic equation
\beq
(1-n_0(x))\rho_{\infty}-\Delta w_{\infty}=\rho_0, \quad w_{\infty}(1-\rho_{\infty})=0,\quad \rho_\infty \in [0,1], \quad w_\infty \geq 0.
\label{eqn: eqn for stationary state}
\eeq
Then we have
\begin{align*}
\norm{\rho(\cdot,t)-\rho_{\infty}}_{L^1(\RR^d)}
\leq &\; \frac{\norm{n_0\rho_0}_{L^1(\RR^d)}}{1-\norm{n_0}_{L^{\infty}(\RR^d)}} e^{(\norm{n_0}_{L^{\infty}(\RR^d)}-1)t},\\
\norm{\eta(\cdot,t)-n_0\rho_{\infty}}_{L^1(\RR^d)}\leq &\; \frac{\norm{n_0\rho_0}_{L^1(\RR^d)}}{1-\norm{n_0}_{L^{\infty}(\RR^d)}} e^{(\norm{n_0}_{L^{\infty}(\RR^d)}-1)t}.
\end{align*}

\end{theorem}
\begin{proof}
Recall that $\eta = \rho\eta$, since $\rho$ is a patch solution.
A direct computation shows that
\[
\int_{\RR^d} \partial_t \rho \,dx =\int_{\RR^d} n_0\rho-\eta\,dx, \quad \int_{\RR^d} \partial_{tt} \rho \,dx = \int_{\RR^d} n_0\partial_t \rho-(n_0\rho-\eta)\,dx.
\]
Therefore,
\[
\int_{\RR^d} \partial_{tt} \rho\,dx +\big(1-\norm{n_0}_{L^{\infty}(\RR^d)}\big)\int_{\RR^d} \partial_t \rho\,dx \leq 0.
\]
Hence,
\[
\int_{\RR^d} \partial_t \rho(x,t)\, dx \leq e^{(\norm{n_0}_{L^{\infty}(\RR^d)}-1)t}\int_{\RR^d} \rho_0 n_0\,dx.
\]
A similar calculation gives
\[
\int_{\RR^d} \partial_t \eta(x,t)\, dx\leq e^{(\norm{n_0}_{L^{\infty}(\RR^d)}-1)t}
\int_{\RR^d} \rho_0n_0\,dx.
\]
Since $\partial_t \rho, \partial_t \eta\geq 0$ almost everywhere, the above bounds imply that there exists $\rho_{\infty}, \eta_{\infty}\in L^1(\RR^d)$ such that
\begin{align*}
&\norm{\rho(\cdot,t)-\rho_{\infty}}_{L^1(\RR^d)}
\leq \frac{\norm{n_0\rho_0}_{L^1(\RR^d)}}{1-\norm{n_0}_{L^{\infty}(\RR^d)}} e^{(\norm{n_0}_{L^{\infty}(\RR^d)}-1)t},\\
&\norm{\eta(\cdot,t)-\eta_{\infty}}_{L^1(\RR^d)}\leq \frac{\norm{n_0\rho_0}_{L^1(\RR^d)}}{1-\norm{n_0}_{L^{\infty}(\RR^d)}} e^{(\norm{n_0}_{L^{\infty}(\RR^d)}-1)t}.
\end{align*}

Finally, it is clear that $0=\lim_{t\to\infty} \partial_t \eta(\cdot,t)=\rho_{\infty}n_0-\eta_{\infty}.$ Hence, $\eta_{\infty}=n_0\rho_{\infty}$.  Since all of the variables are increasing with respect to time, we can pass to the limit in the elliptic equation \eqref{eq:elliptic} to obtain \eqref{eqn: eqn for stationary state}.
\end{proof}



\section{Regularity of the Tumor Patch}
\label{sec:regularity}

In this section we analyze the free boundary regularity of the tumor region $\{\rho=1\}$. For simplicity we only consider patch solutions $\rho$ of the tumor growth model \eqref{eqn: P}.

\medskip

When $\rho_0$ equals the characteristic function $\chi_{\Omega_0}$ with $\Omega_0$ being a compact set, Proposition~\ref{prop:contraction} and Lemma \ref{lem:01} imply that for any $t>0$, $\rho(\cdot,t)= \chi_{\Omega_t}$ for some $\Omega_t$ that increases in time.
Lemma \ref{lem:rho_1_characterization} gives a characterization of $\Omega_t$.
We will focus on the boundary regularity of $\Omega_t$ under suitable geometric assumptions.

\subsection{Reflection geometry}

First, we recall from \cite{FK} that sets having \emph{reflection geometries}, i.e., those satisfying ordering properties when reflected with respect to a family of hyperplanes, have locally Lipschitz boundaries.
For both simplicity and relevance to the case $n_0(x)\geq 1$, we focus on isotropic reflection geometry defined as \emph{$r$-reflection property} below. For more general types of reflection geometries that ensure Lipschitz regularity of the set boundary, see \cite{KKP}.

For a hyperplane $H$ in $\R^d$ with unit normal vector $\nu_H$, the reflection with respect to $H$ is given by
$$
\phi_H(x):= x-2(x-y, \nu_H) \nu_H \hbox{ for some }y \in  H.
$$

\begin{definition}[Definition 3.10, \cite{FK}]
\label{def: set reflection}
For $r>0$, we say a bounded open set $\Omega\subset \R^d$ satisfies  {\it $r$-reflection property} if, for any hyperplane $H$ such that $H^-:= \{x: (x-y)\cdot \nu_H <0,\,y\in H\}$ contains $B_r(0)$, 
$$
\phi_H(\Omega\cap H^+) \subset \Omega \cap H^-.
$$
Here $H^+ : = \phi_H(H^-)$.
\end{definition}

\begin{lemma}[Lemma 3.24, \cite{FK}]\label{star_Lipschitz}
Let $C(x,\theta):= \{y: \langle x,y\rangle  \geq  \cos \theta |x||y|\}$ denote the cone with direction $x$ and angle $\theta \in [0, \pi/2]$.
Suppose $\Omega$ satisfies $r$-reflection property, and that $\Omega$ contains the closure of $B_r(0)$.
Then for all $x\in \partial\Omega$, there is an exterior cone $C(x,\phi_x)$ to $\Omega$ at $x$, such that
$$
x+ C(x, \phi_x) \subset \Omega^c, \hbox{ where } \cos \phi_x = \frac{r}{|x|}.
$$
\end{lemma}

\begin{definition}
Let $\mathcal{P}$ be a family of hyperplanes in $\BR^d$ that do not go through the origin.
$u$ is said to have reflection properties with respect to $\mathcal{P}$, 
if for any $H \in \mathcal{P}$, $u_H \leq u$ in $H^+$.
Here $u_H(x): =u\circ \phi_H(x)$ is the reflection of $u$ with respect to $H$, and $H^+$ is the half-space generated by $H$ that contains the origin.
\end{definition}

Lemma \ref{star_Lipschitz} implies that, if $u$ has reflection properties with respect to $\mathcal{P}$ and there is enough of these hyperplanes in $\mathcal{P}$, then the super level sets of $u$ should have locally Lipschitz boundary.

\subsection{ Lipschitz regularity for the tumor patch}

Using the invariance of the system with respect to reflections, we will show that, under suitable assumptions on $n_0$, if $r$-reflection property holds for $\Omega_0$ initially, it remains to be true for $\Omega_t$ for all $t>0$ (see Proposition \ref{prop: reflection comparison}).
This allows us to prove that $\partial \Omega_t$ enjoys Lipschitz regularity in various typical cases.
For example, if $n_0(x)\geq 1$ on $\BR^d$, we will show that for any initial tumor region $\Omega_0$, $\Omega_t$ has locally Lipschitz boundary for all sufficiently large times: see Corollary \ref{star:shaped}. We will  build on the Lipschitz regularity to later show $C^{1,\alpha}$-regularity of the free boundary, using the elliptic equation that $w$ solves.

\medskip

We first prove the so-called reflection comparison result.
\begin{prop}[Reflection comparison]
\label{prop: reflection comparison}
For a given hyperplane $H$ in $\R^d$, define $\rho_H: = \rho\circ \phi_H$.
Let $H^+$ be one of the half-spaces generated by $H$. 
If $(\rho_0)_H\leq\rho_0$ and $(n_0)_H \leq n_0$ a.e.\;in $H^+$, then $(\rho)_H \leq \rho$, and $(\eta)_H\leq \eta$ a.e.\;in $H^+ \times [0,\infty)$.
\end{prop}

\begin{proof}
Let $w_H: = w\circ \phi_H$ and there holds trivially $w|_H=w_H|_H$.
Thanks to the symmetries of the Laplacian, we have
\[
\rho_H-\Delta w_H=(\rho_0)_H+\eta_H,\quad w_H(1-\rho_H) = 0, \quad \rho_H\in [0,1],\quad w_H\geq 0
\]
almost everywhere in $H^+$.
Thus, the stability estimate in Proposition \ref{prop:contraction} implies that for every $t\in [0,T]$,
\[
\norm{(\rho_H-\rho)_+}_{L^1(H^+\times \{t\})}\leq \norm{(\eta_H-\eta)_+}_{L^1(H^+\times \{t\})}.
\]
Since $\partial_t\eta=n_0\rho-\eta$, it follows that $\partial_t \eta_H=(n_0)_H\rho_H-\eta_H$. Thus,
\[
\begin{split}
&\; \frac{d}{dt}\norm{(\eta_H-\eta)_+}_{L^1(H^+\times \{t\})}+\norm{(\eta_H-\eta)_+}_{L^1(H^+\times \{t\})}\\
\leq &\; \norm{((n_0)_H\rho_H-n_0\rho)_+}_{L^1(H^+\times \{t\})}\\
\leq &\; \norm{(n_0)_H}_{L^{\infty}(H^+)}\norm{(\rho_H-\rho)_+}_{L^1(H^+\times \{t\})}\\
\leq &\;\norm{(n_0)_H}_{L^{\infty}(H^+)}\norm{(\eta_H-\eta)_+}_{L^1(H^+\times \{t\})}.
\end{split}
\]
Now it follows from Gronwall's inequality that, for all $t\in [0,T]$, $\norm{(\eta_H-\eta)_+}_{L^1(H^+\times \{t\})}=0$, and thus $\norm{(\rho_H-\rho)_+}_{L^1(H^+\times \{t\})}=0$.
\end{proof}

The following is a consequence of the reflection comparison and Lemma~\ref{star_Lipschitz}.

\begin{corollary}\label{star:shaped}
Assume that the super-level sets of $n_0$ satisfy the $r$-reflection for some $r>0$. Then the followings hold:
\begin{enumerate}[(a)]
\item Suppose $n_0(x) <1$ on $\BR^d$.
If $\Omega_0$ satisfies $r$-reflection, then so does $\Omega_t$ for all $t>0$. In this case, suppose  $\overline{B_R(0)}\subset \{w_{\infty}>0\}$ for some $R>r$, where $w_\infty$ is defined in Theorem \ref{thm: convergence when n_0 is less than 1}. Then there is $T>0$ such that $\Omega_t$ contains $B_R(0)$ whenever $t>T$.
Consequently, for $t>T$, $\Omega_t$ has uniformly Lipschitz boundary, with Lipschitz constant less than $O(\frac{r}{R-r})$.
\item Suppose $n_0(x)\geq 1$ on $\BR^d$ and that $\Omega_0$ is a bounded open set contained in $B_r(0)$.
Then for any $R>0$, $T(R):= \inf\{t: B_R(0) \subset\Omega_t\}$ is finite. Consequenctly, for any $R>r$, when $t\geq T(R)$, the set $\Omega_t$ has Lipschitz boundary with respect to radial direction with Lipschitz constant less than $O(\frac{r}{R-r})$.
\end{enumerate}
\end{corollary}

\begin{remark}
Note that the $r$-reflection condition does not restrict the shape of either $\Omega_0$ or super-level sets of $n_0$ inside $B_r(0)$.
Hence, we can start with any initial data $\Omega_0\subset B_r(0)$ in both cases of the above corollary, where the evolution of the set may go through topological singularities such as merging of the free boundaries.
The above results state that, given that the initial nutrient $n_0$ is ``well-prepared" outside $B_r(0)$, once $\partial \Omega_t$ moves outside $B_r(0)$, there will be no further topological changes in the evolution, and $\partial \Omega_t$ remains being Lipschitz.
\end{remark}
\begin{remark}
One could relax the assumption on $n_0$ so that for level sets lying between $B_{R}(0)$ and $B_{2R}(0)$, the corresponding super-level set only satisfies $R$-reflection.
That would allow us to study the case, e.g., where level sets of $n_0$ are ellipses.
That may admit possibly non-radial asymptotic shapes for $\Omega_t$.
We leave such discussion to interested readers.
\end{remark}

\begin{proof}
For (a), $\Omega_t$ satisfies $r$-reflection due to Proposition \ref{prop: reflection comparison}. Moreover, both $\rho$ and $\eta$ monotone increases in time and converge to $\rho_{\infty}$ and  $n_0\rho_{\infty}$ in $L^1(\R^d)$, due to Theorem~\ref{thm: convergence when n_0 is less than 1}.  In the patch case we have $\rho_{\infty} = \chi_{\Omega_{\infty}}$, where $\Omega_{\infty}$ is bounded.
This implies the above convergence also holds in $L^p(\BR^d)$ for any $p\in[1,+\infty)$.
It then follows from \eqref{eq:elliptic} and \eqref{eqn: eqn for stationary state} that $w$ then uniformly converges to $w_{\infty}$ as $t\to \infty$.
Hence if we know that $\overline{B_R(0)}\subset \{w_{\infty}>0\}$, from the fact that $w(\cdot,t)$ monotonically increases to converge to $w_{\infty}$, we conclude that $\Omega_t$ contains $\overline{B_r(0)}$ for sufficiently large $t$. Then we can conclude (a) by Lemma~\ref{star_Lipschitz}.

\medskip

For (b), that $T(R)$ being finite follows from comparison with the case $n_0(x) \equiv 1$;
see Theorem \ref{thm:contraction} and Remark \ref{rmk: radial solution} below.
Also note that for any hyperplane $H$ such that $H^-$ contains $B_r(0)$, $\Omega_0 \cap H^+$ is an empty set, and thus its reflected image $\phi_H(\Omega_0\cap H^+)$ is trivially contained by $\Omega_0 \cap H^-$. Thus $\Omega_t$ satisfies $r$-reflection for all $t\geq 0$, and once $t \geq T(R)$ with $R>r$, we can apply Lemma \ref{star_Lipschitz}. 
\end{proof}

Let us now assume that, for a domain $D\subset \R^n\setminus \Omega_0$, $\partial\Omega_t \cap D$ is non-empty but has Lebesgue measure zero.
Then Lemma~\ref{lem:rho_1_characterization} yields that
$\rho(\cdot,t)= \chi_{\{w(\cdot,t)>0\}}$ in $D$, and thus \eqref{eq:elliptic} can be written as
$$
\Delta w = \rho - \eta =  \chi_{\{w>0\}} -\eta \hbox{ in } D.
$$
Moreover, by its definition $\eta= n_0 -n$ is supported only in $\chi_{\{w>0\}}$. Hence, \eqref{eqn: eta equation}-\eqref{eq:elliptic} yields the following $(w,\eta)$-system in $D$:
\beq\label{eqn: W}
\left\{\begin{array}{lll}
\eta_t = (n_0-\eta)\rho  = n_0\chi_{\{w>0\}}  -\eta; \\ \\
\chi_{\{w>0\}} -\Delta w = \eta\chi_{\{w>0\}}.
\end{array}\right. \tag{W}
\eeq

In what follows, we will explore the second equation in \eqref{eqn: W} to prove further regularity of the free boundary, in the setting given in Corollary~\ref{star:shaped}.

\subsection{$C^{1,\alpha}$-regularity of the free boundary}

According to \eqref{eqn: W}, away from the support of $\rho_0$, $w$ solves
$$
(1-\eta)\chi_{\{w>0\}}-\Delta w =0.$$
As long as $\eta<1$ near $\partial \{w>0\}$ and is $C^{\alpha}$, this problem falls into the category of standard obstacle problem, whose singular points feature a blow-up profile of a quadratic polymonial with sub-quadratic error term.
This is impossible if the set $\{w>0\}$ is known to have locally Lipschitz free boundary in $D$.
Hence, if we know a priori that $\Omega_t$ has Lipschitz boundary, then the free boundary $\partial\{w>0\}$ consists of only regular points, i.e., it is $C^{1,\alpha}$; see \cite[Theorem 7.2]{Blank}. This would be the conclusion we will obtain at the end of this section, in Corollary~\ref{cor:C 1 alpha regularity}, with a class of initial data discussed in Corollary~\ref{star:shaped}.

\medskip

To complete this argument, we need to prove that $\eta$ is indeed H\"{o}lder continuous in space. Note that $\eta(x,t)$ starts evolving in time once the set $\Omega_t$ reaches $x$. Regularity of $\eta$ thus is directly related to the dynamics of the set $\Omega_t$, or equivalently that of $\{w(\cdot,t)>0\}$. We will show a version of non-degeneracy for the pressure variable near the boundary (see Proposition \ref{pressure:nondeg}), which further implies non-degeneracy of the propagation speed of the boundary, ensuring that the tumor patch reaches nearby points with small time difference.

\medskip

To study the dynamics of the tumor patch, we will need a direct comparison principle for $(\rho,p)$, with barriers with a fixed $n$, as follows.

\begin{lemma}\label{direct_cp}
Let $(\rho, p, n)$ solve the original system \eqref{eqn: P}, and suppose $(\bar{\rho}, \bar{p})$ weakly solves
$$
\partial_t\bar{\rho}- \Delta \bar{p} \leq n\bar{\rho} \;\hbox{ in } D \times [t_0, t_1], \quad
$$
 with $\bar{p}\in L^2([0,T];H^1(\RR^d))$ such that $\bar{\rho}=\sgn_+(\bar{p})$.
If $\bar{\rho} \leq \rho$ at $t=t_0$,  $\bar{p} \leq p$ on $\partial D\times [t_0,t_1]$,  and $\partial_t\bar{\rho}\geq 0$, then $\bar{\rho} \leq \rho$ and $\bar{p} \leq p$ almost everywhere on $D\times [t_0, t_1]$.
\end{lemma}
\begin{remark}
We expect the result to still hold if one drops the requirement that $\partial_t\bar{\rho}\geq 0$ and $\bar{\rho}=\sgn_+(\bar{p})$. Nonetheless, these assumptions make the proof substantially easier and the above statement is sufficiently strong for our purposes.
\end{remark}

\begin{remark}
Let us note that this result is closely related to the comparison and uniqueness statements Lemma \ref{lem: comparison} and Proposition \ref{prop: uniqueness}.   Nonetheless, neither statement directly applies to prove the above result.   Lemma \ref{lem: comparison} assumes that one already has an ordering for the source terms,   while Proposition \ref{prop: uniqueness} proves uniqueness through stability rather than comparison.  On the other hand, it is very likely that one could obtain a strengthened version of this result by appropriately tweaking the argument in Lemma \ref{lem: comparison}.
\end{remark}

\begin{proof}
Let $\bar{w}(x,t)=\int_{t_0}^t \bar{p}(x,s)\, ds$.  Since $\bar{\rho}$ is increasing and $\bar{\rho}=\sgn_+(\bar{p})$, it follows that $\bar{\rho}=\sgn_+(\bar{w})$ for any $t>t_0$.
Hence, integrating in time, we see that $\bar{w}$ weakly solve
\[
\sgn_+(\bar{w})(x,t)-\Delta \bar{w}(x,t)\leq \bar{\rho}(x,t_0)+ \int_{t_0}^t n\sgn_+(\bar{w})(x,s)\,ds.
\]
If we set $w(x,t)=\int_{t_0}^t p(x,s)\, ds$ with abuse of notations, we recall that the original system solves
\[
\rho(x,t)-\Delta w(x,t)=\rho(x,t_0)+\int_{t_0}^t n\rho(x,s)\,ds, \quad w(1-\rho)=0,\quad \rho\in [0,1],\quad w\geq 0
\]
almost everywhere in $D\times [t_0, t_1]$. 
Then we argue as in the proof of Proposition \ref{prop:contraction}.
Fix $\delta>0$ and let $f_{\delta}:\RR\to\RR$ be a smooth increasing function such that $f_{\delta}(a)=0$ if $a\leq 0$ and $f_{\delta}(a)=1$ if $a\geq \delta$. Fix some time $t>t_0$. Taking the difference of the above two formulas and integrating against $f_{\delta}(\bar{w}-w)$ along $D\times [t_0, t]$, we see that
\[
\begin{split}
&\; \int_{D\times [t_0,t]} (\sgn_+(\bar{w})-\rho)f_{\delta}(\bar{w}-w)+f'_{\delta}(\bar{w}-w)|\nabla (\bar{w}-w)|^2\,dx \,ds \\
\leq &\; \|n_0\|_{L^\infty(\BR^d)}
\int_{t_0}^{t} \norm{(\sgn_+(\bar{w})-\rho)_+}_{L^1(D\times [t_0, s])}\, ds,
\end{split}
\]
where we have used the fact that $f_{\delta}\in [0,1]$ almost everywhere, that $f_{\delta}(\bar{w}-w)$ vanishes almost everywhere on $\partial D\times [t_0, t_1]$, and that $\bar{\rho}\leq \rho$ at $t = t_0$.
Since $\sgn_+(\bar{w})$ and $\rho$ only take the values $0$ or $1$ almost everywhere, we have
\[
(\sgn_+(\bar{w})-\rho)_+=(\sgn_+(\bar{w})-\rho)\sgn_+(\bar{w}-w).
\]
Hence, we can send $\delta\to 0$ in the above inequality to obtain
\[
\norm{(\sgn_+(\bar{w})-\rho)_+}_{L^1(D\times [t_0, t])} \leq
\|n_0\|_{L^\infty(\BR^d)} \int_{t_0}^{t} \norm{(\sgn_+(\bar{w})-\rho)_+}_{L^1(D\times [t_0, s])}\, ds.
\]
Now Gronwall's inequality implies that $\bar{\rho}=\sgn_+(\bar{w})\leq \rho$ almost everywhere in $D\times [t_0, t_1]$.

It remains to prove $\bar{p}\leq p$ almost everywhere in $D\times [t_0, t_1]$.
Choose some $\psi\in L^2([t_0, t_1];H^1(D))$ such that $\psi\geq 0$ almost everywhere in $D\times [t_0, t_1]$, $\psi = 0$ at $t = t_1$, $\psi(1-\bar{\rho})=0$, and $\psi|_{\partial D\times [t_0, t_1]}=0$.
If we fix $\epsilon>0$ and let $\psi^{\epsilon}(x,t)=\epsilon^{-1}\int_{t}^{\min(t_1, t+\epsilon)} \psi(x,s)\, ds$,
then $\psi^{\epsilon}$ is a valid test function for the weak equation $\partial_t \bar{\rho}-\Delta \bar{p}\leq \bar{\rho} n$ since $\psi^\epsilon = 0$ at $t=t_1$.
Hence,
\[
\int_{D\times [t_0, t_1]} -\bar{\rho}\partial_t \psi^{\epsilon}+\nabla \psi^{\epsilon}\cdot \nabla \bar{p}-\psi^{\epsilon}\bar{\rho} n \,dx \,dt
\leq
\int_D \bar{\rho}(x,t_0)\psi^{\epsilon}(x,t_0)\,dx.
\]
Note that for almost every $t\in [t_0,t_1]$,
\[
\bar{\rho}\partial_t\psi^{\epsilon}
=\bar{\rho}(x,t)\frac{\psi(x, \min\{t_1,t+\epsilon\})-\psi(x,t)}{\epsilon}\leq \frac{\psi(x, \min\{t_1,t+\epsilon\})-\psi(x,t)}{\epsilon}=\partial_t \psi^{\epsilon},
\]
where we use $\psi(1-\bar{\rho})=0$ and the non-negativity of $\psi^{\epsilon}$ to justify the inequality. Therefore,
 \[
\int_{D\times [t_0, t_1]} \nabla \psi^{\epsilon}\cdot \nabla \bar{p}-\psi^{\epsilon}\bar{\rho} n\,dx \,dt
\leq
\int_D (\bar{\rho}(x,t_0)-1)\psi^{\epsilon} (x,t_0)\,dx
\leq 0.
\]
Sending $\epsilon\to 0$ and once again using $\psi(1-\bar{\rho})=0$, we see that
\begin{equation*}
\int_{D\times [t_0, t_1]} \nabla \psi\cdot \nabla \bar{p}-\psi n\,dx \,dt \leq  0.
\end{equation*}
Since $\bar{\rho}\leq \rho$, we also have $\psi(1-\rho)=0$. Therefore the complementarity condition, Proposition \ref{prop:complementarity}, implies that
\begin{equation*}
\int_{D\times [t_0, t_1]} \nabla \psi\cdot \nabla p-\psi n\,dx\,dt
= 0.
\end{equation*}
Combining the above two formulas yields
\begin{equation*}
\int_{D\times [t_0, t_1]} \nabla \psi\cdot \nabla (\bar{p}-p)\,dx\,dt
\leq 0.
\end{equation*}
Choose $\psi=\omega(t)(\bar{p}-p)_+$, where $\omega(t)$ is a smooth non-negative function such that $\omega >0$ on $[t_0,t_1)$ and $\omega(t_1) = 0$.
It is clear that this choice satisfies our assumptions on $\psi$.
Hence, we obtain that
\[
\int_{D\times [t_0, t_1]} \omega(t) |\nabla (\bar{p}-p)_+|^2\,dx \,dt \leq  0.
\]
Since $\bar{p}\leq p$ almost everywhere on $\partial D\times [t_0, t_1]$, we obtain $\bar{p}\leq p$ almost everywhere on $D\times [t_0, t_1]$.
\end{proof}

Based on above comparison principle, we will build a radial barrier $(\bar{\rho}, \bar{p})$ to compare with  $(\rho,p)$, to show that the pressure support spreads with a uniform rate. To construct the barrier it is useful to  recall  Dahlberg's lemma:

\begin{lemma}[Dahlberg's lemma, \cite{dahlberg}]\label{lem:D}
Let $u_1, u_2$ be two non-negative harmonic functions in $D\subset \R^d$  of the form
$$
D=\big\{(x',x_n)\in \R^{n-1}\times  \R: |x'|  <2, |x_n|< 2M, x_n> f(x') \big\}
$$
with a Lipschitz function $f$ with Lipschitz constant less than $M$ and $f(0)=0$. Assume further that $u_1=u_2=0$ along the  graph of  $f$. Then there exist constants $C_1, C_2$ only depending on $M$, such that
$$
0< C_1  \leq \dfrac{u_1(x', x_n)}{u_2(x', x_n)} \cdot\dfrac{u_2(0,M)}{u_1(0,M)}  \leq C_2
$$
in the smaller domain
$$
D_{1/2}= \big\{|x'|<1, |x_n|<M, x_n>f(x')\big\}.
$$
\end{lemma}

\medskip

Let us now define  $C_{\theta}:= C(-e_n,\theta) = \{x_n<f(x')\}$ with $f(x'):= -\cot\theta |x'|$, and let $h$ solve
$$-\Delta h = \delta_{\{x=-3e_n\}} \hbox{ in }C_{\theta}, \quad h=0 \hbox{ on }\{x_n=f(x')\}.
$$
Then $h$ has a polynomial growth from the origin, namely,
\begin{equation}\label{growth}
\begin{split}
&\mbox{there exists some }k = k(\theta)>0\mbox{ that decreases in }\theta,\mbox{ such that }\\
&h(x) \simeq (f(x')-x_n)_+^k \hbox{ in }  C_{\theta}\cap B_2(0).
\end{split}
\end{equation}
Let $\theta_d$ be the angle such that $k(\theta_d)= 2$, where $h$ has quadratic growth near the origin.
For instance, $\theta_2=\pi/4$, and $\theta_d$ increases as $d$  increases.

\begin{prop}[Nondegeracy of the pressure]\label{pressure:nondeg}
Suppose $x_0 \in \partial\{p(\cdot,t_0)>0\}$, and suppose that $\{p(\cdot,t_0)>0\}$ contains $(x_0+C_{\theta}) \cap B_2(x_0)$ with $\theta>\theta_d\geq \pi/4$.
Let $\bar{n} := \min \{n(x,t_0): |x-(x_0-2e_n)| < 1\}$.
Then there exists a dimensional constant $C>0$ such that the following holds: for any $0<r<1/C$, we have
$$
B_r(x_0) \subset \left\{p\left(\cdot,t_0+ \frac{Cr^\alpha}{\bar{n}}\right)>0\right\},
$$
where $\alpha=2-k(\theta)>0$, with $k(\th)$ given in \eqref{growth}.
\end{prop}

\begin{proof}
Throughout the proof, we use $C$ to denote various dimensional constants.  We may assume $x_0 = 0$ by shifting the coordinates properly.
By our assumption $\theta >\pi/4$, we have $B_1(-2e_n)  \subset \{x_n < f(x')\}$, where $f(x'): = -\cot\th |x'|$ as above. Since $-\Delta p(\cdot,t_0) \geq  \bar{n}$ in $B_1(-2e_n)$, it follows that $p(x,t_0) \geq C\bar{n}$ in $B_{1/2}(-2e_n)$.

\medskip

Next consider a harmonic function $q$ in the domain $\{x_n<f(x')\} \cap (B_2(0) \backslash B_{1/2}(-2e_n))$,  with boundary data
$$q=0 \hbox{ on } \{x_n=f(x')\}\cup \partial B_2(0) \hbox{ and } q=C\bar{n} \hbox{  on }\partial B_{1/2}(-2e_n).
$$
 Since $p(\cdot,t_0)$ is superharmonic, it follows that $p(\cdot,t_0) \geq q$ in $B_{1/2}(0) \cap \{x_n<f(x')\}$.

\medskip

To obtain explicit lower bound for $q$ near $\{x_n = f(x')\}$,  let us compare $q$ with $h$ given in \eqref{growth}. By Dahlberg's lemma, $q \geq C\bar{n} h$ in $\{x_n<f(x')\} \cap B_1(x_0)$. In particular, it follows that
\begin{equation}\label{lower_bd_0}
p(x,t_0)  \geq C\bar{n} (f(x')-x_n)_+^{2-\alpha} \hbox{ in } \{x_n<f(x')\}\cap B_1(x_0).
\end{equation}
Since $\{p(\cdot,t)>0\}$ increases in time, by the same logic and the property of the $n$-equation $n_t = -n\rho$, we have
\begin{equation}\label{lower_bd}
p(x,t) \geq C\bar{n} e^{-(t-t_0)} (f(x')-x_n)_+^{2-\alpha} \hbox{ in } B_1(x_0)\times [t_0, \infty).
\end{equation}

Now we will use Lemma \ref{direct_cp}  to estimate the time it takes for the positive set of $p$ to fully cover $B_r(x_0)$.
Let $x_r:= x_0 -re_n$ and $D:= \{ x: |x-x_r| > \frac{r}{4}\}$. Let $\bar{p}$ be a barrier solving
$$
\left\{
\begin{array}{lll}
-\Delta \bar{p} = 0 &\hbox{ in }& \frac{r}{4} <|x-x_r| <r(t), \\ \\
\bar{p} = C\bar{n}e^{- (t-t_0) }r^{2-\alpha} &\hbox{ on } & \{|x-x_r|=\frac{r}{4}\},\\ \\
\bar{p}=0 &\hbox{ on } &\{|x-x_r| = r(t)\}.
\end{array}\right.
$$
If we set $r(t)$ so that $r'(t) \leq |D\bar{p}|$ on $\{|x-x_r|=r(t)\}$, then $\bar{p}$ will solve
$$
\bar{\rho}_t - \Delta \bar{p} \leq 0 \hbox{ in } D \times [t_0, \infty),
$$
where $\bar{\rho} = \chi_{\{\bar{p}>0\}}$.
If in addition $r(t_0) = \frac{r}{2}$ so that $\{\bar{p}(\cdot,t_0) >0\} \subset \{p(\cdot,t_0)>0\}$,  we will apply Lemma~\ref{direct_cp} and \eqref{lower_bd} to conclude that $\bar{p} \leq p$ for $t>t_0$, which will then yield a lower bound for the time it takes for the positive set of $p$ to include $B_r(x_0)$.

\medskip

Now we choose a specific $r(t)$.
Observe that $|D\bar{p}| \geq C \bar{n}e^{- (t-t_0) } r^{1-\alpha}$ on $\{|x-x_r|=r(t)\}$ as long as $r(t)\in [\f12 r, 2r]$.
Hence we can set
$$
r(t) := \frac{r}{4} + C\bar{n}\big(1-e^{- (t-t_0) }\big)r^{1-\alpha}.
$$
Now we conclude, since $r(t)=2r$ when $\bar{n}(1-e^{- (t-t_0) }) \sim r^{\alpha}$, i.e., when $t-t_0 \sim \frac{r^{\alpha}}{\bar{n}}$.
\end{proof}

\begin{corollary}
\label{cor:C 1 alpha regularity}

Suppose that, given a point $x_0\in \partial\{w(\cdot,t_1)>0\}$ with $t_1>0$,  the set $\partial\{w(\cdot,t)>0\}$ is a Lipschitz graph in $B_2(x_0)$ for all $0<t\leq t_1$ with respect to a fixed direction.
Further suppose that the Lipschitz constants of the graphs are all smaller than a dimensional constant. Then the followings hold:
\begin{enumerate}[(a)]
\item $\eta(\cdot,t_1) \in C^{\alpha}(B_1(x_0))$ ;
\item $\partial\{w(\cdot,t_1)>0\}$ is  $C^{1,\alpha}$ in $B_1(x_0)$.
\end{enumerate}
Here the $C^{\alpha}$ and $C^{1,\alpha}$ norms only depend on $\bar{n}:=  \min_{ B_2(x_0)} n(\cdot,t_1)$.
\end{corollary}

\begin{proof}
Define $T_x:= \inf\{t: w(x, t) >0\}$ for all $x\in \R^d$. Then for $x,y\in B_1(x_0)\cap \{w(\cdot,t_1)>0\}$,  Proposition \ref{pressure:nondeg} implies that, with $r$ being smaller than a universal constant, $T_y \leq T_{x} + \frac{Cr^{\alpha}}{\bar{n}}$ as long as $|x-y|<r$.
Hence we have
$$
|T_x - T_y| \leq \frac{C|x-y|^{\alpha}}{\bar{n}} \hbox{ for all } x,y \in B_1(x_0).
$$
Note that $\eta(x,t) = n_0(x)(1-\exp^{-(t-T_x)_+})$, and $n_0$ is H\"{o}lder continuous by our assumption. Thus we conclude (a) by deriving that
$$
|\eta(x,t_1) - \eta(y,t_1)| \leq \frac{C|x-y|^{\alpha}}{\bar{n}} \hbox{ for all } x,y  \in B_1(x_0).
$$
(b) then follows from \cite[Theorem 7.2]{Blank}.
\end{proof}

\begin{remark}
Once we obtain $C^{1,\alpha}$-regularity of the free boundary $\partial\{w(\cdot,t)>0\}$,  the interior cone angle $\theta$ given in Proposition \ref{pressure:nondeg} can be chosen as close to $\frac{\pi}{2}$ as desired, in a  smaller scale. We can thus improve the regularity of $\eta$ to $C^{1-\e}$ for any $\e>0$, which in turn improves the free boundary  regularity to $C^{1,1-\e}$ for any $\e>0$.
\end{remark}

Combining the above results with Corollary~\ref{star:shaped}, we arrive at the following conclusion.

\begin{corollary}\label{cor:1}
Suppose $n_0$ is $C^1$, and its super-level sets satisfy $r$-reflection for some $r>0$.
Let $\Omega_0$ be an open bounded set in $\R^d$ contained in $B_r(0)$.
Let $T(R)$ be defined as in Corollary \ref{star:shaped}.
Then the followings hold for any $0<\alpha<1$ and any $R\geq Cr$ with $C>1$ being a universal constant:
\begin{enumerate}[(a)]
\item If $n_0(x)\geq 1$ on $\BR^d$, then $\partial\{\rho(\cdot,t)=1\} = \partial\{w(\cdot,t)>0\}$ is uniformly $C^{1,\alpha}$ in a unit neighborhood for any finite time range within $[T(R), \infty)$.
\item If $n_0(x)<1$ on $\BR^d$, the same holds in for any finite time range within  $[T(R), \infty)$ if $R$ additionally satisfies $B_R(0)\subset \{w_{\infty}>0\}$.
    Here $w_\infty$ is defined in Theorem \ref{thm: convergence when n_0 is less than 1}.
\end{enumerate}
\end{corollary}

\medskip

Both Corollary \ref{cor:C 1 alpha regularity} and Corollary \ref{cor:1} only apply to finite time ranges.
It is natural to ask what can be said about uniform regularity of the free boundary up to $t = +\infty$.
This is a nontrivial question due to the possible decay of $\bar{n}$ (defined in Corollary~\ref{cor:C 1 alpha regularity}) as $t$ tends to infinity.  Note that, roughly speaking, at a boundary point $x_0$, $\bar{n}\geq n_0 e^{-t}$ if the free boundary does not move significantly.
When $n_0<1$, this bound is close to optimal since the tumor patch converges to a bounded set. When $n_0 \geq 1$, however, it is possible to improve this bound by comparison with radial barriers. We will discuss this in Section \ref{sec: uniform boundary regularity}, under the additional assumption that $n_0$ is constant.

%
%

\section{The Constant $n_0$ Case}
\label{sec: master dynamics}
In this section, we further look into the case when $n_0$ is a positive constant, while still inheriting the assumptions that $b = D = 0$ and $\rho_0$ is a patch with compact support.
%
%
What is special and surprising in this case is that the dynamics of the system \eqref{eqn: P} can be fully characterized by simpler parameter-free and nutrient-free model problems, which produce the so-called \emph{master dynamics}.
As an application of that, we will address the question of uniform regularity of the free boundary $\partial\Omega_t$.

Recall that for patch solutions, we have $\eta \rho = \eta$.
Then $\eta$ solves (c.f.\;\eqref{eqn: eta equation})
\[
\partial_t \eta +\eta= n_0\rho,\quad \eta(x,0) = 0.
\]
This gives
\beq
\eta(x,t) = n_0 \int_0^t e^{-(t-\tau)}\rho(x,\tau)\,d\tau,
\label{eqn: eta formula}
\eeq
and therefore,
\beq
n(x,t) = n_0 - n_0 \int_0^t e^{-(t-\tau)}\rho(x,\tau)\,d\tau.
\label{eqn: n solution formula}
\eeq
This holds even for non-constant $n_0$.
As a result, in the following, we will focus on the $\rho$-evolution in  \eqref{eqn: P}.
Let us also recall the elliptic formulation of the $\rho$-equation in $\BR^d$ that is derived from \eqref{eqn: P}
\beq
\rho -\Delta w = \rho_0 + \eta, \quad w(1-\rho) = 0,\quad \rho\in[0,1],\quad w\geq 0.
\label{eqn: original eqn}
\eeq
Here $\rho_0$ is a patch with compact support.

\subsection{Two master dynamics}
We first prove a growth law of total mass of $\rho$ and a generalization of it.
The latter will be the key of proving the master dynamics.

\begin{lemma}
\label{lem: general test function}
Define (c.f.~Theorem \ref{thm:contraction})
\beq
m(t) = 1+n_0 \int_0^t e^{(n_0-1)\tau}\,d\tau =
\begin{cases}
         \f{n_0e^{(n_0-1)t}-1}{n_0-1}, & \mbox{if } n_0\neq 1, \\
         t+1, & \mbox{otherwise}.
       \end{cases}
\label{eqn: volume}
\eeq

\begin{enumerate}[(a)]
\item
For any $t\geq 0$,
\[
\int_{\BR^d} \rho(x,t)\,dx = m(t)\int_{\BR^d} \rho_0(x)\,dx.
\]
\item
For any arbitrary smooth function $g = g(x)$ in $\BR^d$, we have for any $t\geq 0$,
\beq
\begin{split}
\int_{\BR^d} \rho(x,t) g(x)\,dx
=&\;
m(t)  \int_{\BR^d} \rho_0(x) g(x) \,dx\\
&\;+\int_{\BR^d} w(x,t)\Delta g(x) \,dx
+ \int_0^t m'(t-\tau)
\left[\int_{\BR^d} w(x,\tau)\Delta g(x) \,dx\right] d\tau.
\end{split}
\label{eqn: general test function}
\eeq
In particular, if $g$ is harmonic in an open neighborhood of the support of $\rho(x,t)$, then
\beq
\int_{\BR^d} \rho(x,t) g(x)\,dx
= m(t) \int_{\BR^d} \rho_0(x) g(x) \,dx.
\label{eqn: quadrature identity}
\eeq
Combined with (a), this implies that when $\rho$ is a patch solution, the average of any harmonic $g$ on $\{\rho(\cdot,t)=1\}$ is time-invariant.
\end{enumerate}

\begin{proof}

Take an arbitrary smooth function $g = g(x)$ in $\BR^d$.
%
We integrate \eqref{eqn: original eqn} in $\BR^d$ against $g(x)$ and use \eqref{eqn: eta formula} to find that
\[
\int_{\BR^d} \rho(x,t) g(x)\,dx 
= 
\int_{\BR^d} \rho_0 g + w(x,t)\Delta g(x) \,dx + n_0 \int_0^t e^{-(t-\tau)} \int_{\BR^d} \rho(x,\tau) g(x)\,dx \,d\tau.
\]
Solving this integral equation, we obtain
\beqo
\begin{split}
\int_{\BR^d} \rho(x,t) g(x)\,dx
= &\;
\int_{\BR^d} \rho_0 g + w(x,t)\Delta g(x) \,dx \\
&\;
+ n_0 \int_0^t e^{(n_0-1)(t-\tau)}
\int_{\BR^d} \rho_0 g + w(x,\tau)\Delta g(x) \,dx \,d\tau\\
= &\;
\left(1+n_0 \int_0^t e^{(n_0-1)(t-\tau)}\,d\tau\right)
\int_{\BR^d} \rho_0 g \,dx \\
&\;+\int_{\BR^d} w(x,t)\Delta g(x) \,dx
+ n_0 \int_0^t e^{(n_0-1)(t-\tau)}
\left[\int_{\BR^d} w(x,\tau)\Delta g(x) \,dx\right] d\tau.
\end{split}
\eeqo
Taking $g(x) = 1$ yields \eqref{eqn: volume}.
That in turn implies \eqref{eqn: general test function}.
\end{proof}
\begin{remark}\label{rmk: time invariance of mass center}
Taking $g(x) = x_i$ $(i = 1,\cdots, d)$ in \eqref{eqn: quadrature identity}, we find the center of mass of $\rho(x,t)$ is time-invariant.
\end{remark}
\begin{remark}
\label{rmk: radial solution}
Suppose $\rho_0 = \chi_{B_{r_0}(0)}$.
Then thanks to \eqref{eqn: volume}, Remark \ref{rmk: time invariance of mass center}, and the uniqueness (c.f.~Proposition \ref{prop: uniqueness} and Theorem \ref{thm:contraction}),
\[
\rho(x,t) = \chi_{B_{r(t)}(0)}, \mbox{ where }r(t) = m(t)^\f1d r_0.
\]
Here $m(t)$ is given in \eqref{eqn: volume}.
In this paper, such radial solutions are repeatedly used as barriers for comparison.
\end{remark}
\end{lemma}

\begin{lemma}
\label{lem: identity for gamma convolve with patches}
Let $\Gamma$ be the fundamental solution of $-\Delta$ in $\BR^d$, i.e., $-\Delta \Gamma = \delta_{x=0}$.
For any $x\in \{w (\cdot,t)= 0\}$,
\beq
\big(\Gamma *\rho(\cdot,t)\big)(x) = m(t) \big(\Gamma*\rho_0\big)(x).
\label{eqn: identity for gamma convolve with patches}
\eeq
\begin{proof}
We introduce a smooth mollifier $\varphi\in C_0^\infty(\BR^d)$ such that $\varphi\geq 0$, $\varphi$ is radially symmetric, and $\int_{\BR^d}\varphi = 1$.
With $\e>0$, define $\varphi_\e(x) := \e^{-d}\varphi(\f{x}{\e})$.
Let $\Gamma_\e := \Gamma*\varphi_\e$, which is clearly smooth in $\BR^d$ and which satisfies $-\Delta \Gamma_\e = \varphi_\e$.
Applying Lemma \ref{lem: general test function} with $g(\cdot)=\Gamma_\epsilon(x-\cdot)$, we find that, for any $x\in \BR^d$,
\beqo
\begin{split}
\big(\Gamma_\e*\rho(\cdot,t)\big)(x) = &\; \int_{\BR^d} \rho(y,t) \Gamma_\e(x-y)\,dy\\
= &\;
 m(t) \int_{\BR^d} \rho_0(y) \Gamma_\e(x-y) \,dy\\
&\; +\int_{\BR^d} w(y,t)\Delta_y \Gamma_\e(x-y) \,dy
+ \int_0^t m'(t-\tau)
\left[\int_{\BR^d} w(y,\tau)\Delta_y \Gamma_\e(x-y) \,dy\right] d\tau\\
= &\; m(t) \big(\Gamma_\e*\rho_0\big)(x)
- \varphi_\e*\left[w(\cdot,t)
+ \int_0^t m'(t-\tau) w(\cdot,\tau) \,d\tau \right].
\end{split}
\eeqo
Then we take $\e\to 0$.
Since $w(x,t)$ is non-decreasing in time (c.f.~Proposition \ref{prop:comparison} and the fact $\eta$ is non-decreasing in time), the term in the bracket in the last line can be dominated by $ m(t) w(\cdot,t)$.
Hence, for any $x\in \BR^d$ satisfying
\beq
\lim_{\e\to 0}\big(\varphi_\e*w(\cdot,t)\big)(x) = 0,
\label{eqn: Lebesgue point condition with value zero}
\eeq
we can show \eqref{eqn: identity for gamma convolve with patches} by using the spatial continuity of $\Gamma*\rho$ and $\Gamma*\rho_0$.
Since $w(\cdot,t)$ is continuous, the condition \eqref{eqn: Lebesgue point condition with value zero} holds in the set $\{w (\cdot,t)= 0\}$.
This completes the proof.
\end{proof}
\end{lemma}

\begin{lemma}
\label{lem: equivalent source}
Let $m(t)$ be given by \eqref{eqn: volume}.
Consider the elliptic equation in $\BR^d$
\beq
\tilde{\rho} - \Delta\tilde{w} = m(t) \rho_0,\quad \tilde{w}(1-\tilde{\rho}) = 0,\quad  \tilde{\rho}\in[0,1],\quad \tilde{w}\geq 0.
\label{eqn: new eqn}
\eeq
Then $\tilde{\rho} = \rho$ for all $t\geq 0$ and almost everywhere in space.

\begin{proof}
By \eqref{eqn: eta formula} and Lemma \ref{lem: identity for gamma convolve with patches}, For all $x\in \{w (\cdot,t)= 0\}$, 
\[
\big(\Gamma*\eta(\cdot,t)\big)(x)
=
n_0 \int_0^t e^{-(t-\tau)} m(\tau) \,d \tau \cdot (\Gamma*\rho_0)(x) = n_0 \left( m(t) - e^{(n_0-1)t}\right) (\Gamma*\rho_0)(x).
\]
Here we calculated by integration by parts that
\beqo
\begin{split}
\int_0^t e^{-(t-\tau)}  m(\tau) \,d\tau
= &\; e^{-t} \left( e^{t}  m(t) -1
- \int_0^t e^\tau m'(\tau) \,d\tau\right)\\
= &\; m(t) - e^{-t}
- e^{-t} \int_0^t e^\tau \cdot n_0e^{(n_0-1)\tau}\,d\tau\\
= &\; m(t) - e^{(n_0-1)t}.
\end{split}
\eeqo

In the view of right-hand sides of \eqref{eqn: original eqn} and \eqref{eqn: new eqn}, let
\[
\Phi : = \Gamma*\left(\rho_0+\eta(\cdot,t)- m(t) \rho_0\right).
\]
One can check that, for all $x\in \{w(\cdot,t) = 0\}$,
\[
\Phi(x) = \left[(n_0-1) m(t) - n_0e^{(n_0-1)t}+1\right] (\Gamma*\rho_0)(x) = 0.
\]

Now we claim that
\[
\tilde{\rho} = \rho,\quad \tilde{w} = w -\Phi,
\]
gives a solution of \eqref{eqn: new eqn}.
One only has to verify that
$\tilde{w}(1-\tilde{\rho}) = 0$.
From what has been proved, whenever $w = 0$, we have $\Phi = 0$ and thus $\tilde{w} = 0$.
On the other hand, $\tilde{\rho} = 1$ whenever $\rho = 1$.
Therefore, $\tilde{w}(1-\tilde{\rho}) = 0$ holds whenever $w(1-\rho) = 0$.
This justifies the claim.

Since the solution of \eqref{eqn: new eqn} is unique (c.f.~Proposition \ref{prop:comparison}), we conclude that $\tilde{\rho} = \rho$.
\end{proof}
\end{lemma}

This leads to the following equivalent characterization of $\{\rho(\cdot,t)\}_t$. 

\begin{prop}[Master dynamics I]
\label{prop: master dynamics}
Let $\rho_* = \rho_*(x,t)$ and $p_* = p_*(x,t)$ be a weak solution (in the sense of Definition \ref{def: weak solution p equation}) of
\beq
\partial_t \rho_* -\nabla\cdot (\rho_*\nabla p_* ) = \rho_0,\quad \rho_*\leq 1,\quad p_*\in P_\infty(\rho_*),\quad \rho_*|_{t = 0} = \rho_0.
\label{eqn: master dynamics}
\eeq
Then for any given $n_0>0$, $\{\rho(x,t)\}_t$ defined by \eqref{eqn: eta formula} and \eqref{eqn: original eqn} (or equivalently, by \eqref{eqn: P}) satisfies
\[
\rho(x,t) = \rho_*(x, m(t) -1 )\mbox{ for all }t\geq 0,
\]
where $m(t)$ is defined by \eqref{eqn: volume}.
$\{\rho_*(x,t)\}_{t\geq 0}$ is thus called the first master dynamics.

In particular, when $n_0 = 1$, $\rho(x,t) = \rho_*(x,t)$.
\end{prop}

We stress that this is a highly non-trivial property of the model under the given assumptions.
It cannot be obtained by a nonlinear change of the time variable.

Following the spirit of Lemmas \ref{lem: general test function}, \ref{lem: identity for gamma convolve with patches}, and \ref{lem: equivalent source},
we may derive a second equivalent characterizations of $\{\rho(x,t)\}_{t}$.
It will be particularly helpful for understanding long-time behavior of $\rho$ when $n_0\geq 1$ because it comes with a suitable spatial re-scaling.

\begin{prop}[Master dynamics II]
\label{prop: master dynamics 2}
Let $\rho_\dag = \rho_\dag(x,t)$ and $p_\dag = p_\dag(x,t)$ be a weak solution of
\beq
\partial_t \rho_\dag -\nabla \cdot \Big[\rho_\dag \nabla \big( p_\dag +V(x)\big)\Big] = 0,\quad \rho_\dag \leq 1,\quad p_\dag\in P_\infty(\rho_\dag),\quad  \rho_\dag|_{t=0} = \rho_0,
\label{eqn: drift under potential}
\eeq
where $V(x) = \frac{|x|^2}{2d}$.
Then for any given $n_0>0$, $\{\rho(x,t)\}_t$ defined by \eqref{eqn: eta formula} and \eqref{eqn: original eqn} (or equivalently, by \eqref{eqn: P}) satisfies
\[
\rho(x,t) = \rho_\dag \left( m(t)^{-\f1d} x,\, \ln m(t) \right)\mbox{ for all }t\geq 0,
\]
where $m(t)$ is defined by \eqref{eqn: volume}.
Thus, $\{\rho_\dag(x,t)\}_{t\geq 0}$ is called the second master dynamics.

\begin{proof}
Following a similar argument as in Lemmas \ref{lem: general test function}, \ref{lem: identity for gamma convolve with patches}, and \ref{lem: equivalent source}, we can prove that, if one defines
\[
\tilde{\tilde{\rho}}-\Delta \tilde{\tilde{w}} = \rho_0 + \int_0^t \f{m'(\tau)}{m(\tau)} \tilde{\tilde{\rho}}(\cdot,\tau)\,d\tau,\quad
\tilde{\tilde{w}} \big(1-\tilde{\tilde{\rho}}\big) = 0,
\]
then $\tilde{\tilde{\rho}} = \rho$ for all $t\geq 0$.
Indeed,
it suffices to show that, on $\{\tilde{\tilde{w}} (\cdot,t)= 0\}$,
\beqo
\big(\Gamma *\tilde{\tilde{\rho}}(\cdot,t)\big)(x)
=
m(t) \big(\Gamma*\rho_0\big)(x),
\eeqo
and thus
\[
\Gamma*\left(\rho_0 + \int_0^t \f{m'(\tau)}{m(\tau)}\tilde{\tilde{\rho}}(\cdot,\tau)\,d\tau
- m(t) \rho_0\right) \equiv 0
\]
on the same set.
We skip the details.

Hence, $\rho$ satisfies
\beq
\partial_t \rho  -\nabla\cdot \big(\rho \nabla \tilde{\tilde{p}}\big) = \f{m'(t)}{m(t)}\rho,\quad \rho\leq 1,\quad \tilde{\tilde{p}}\in P_\infty(\rho),\quad \rho|_{t = 0} = \rho_0.
\label{eqn: a third equivalent dynamics}
\eeq
%
%
It is then straightforward to verify that, if $\rho_\dag$ and $p_\dag$ solve \eqref{eqn: drift under potential}, then
\begin{align*}
&\rho(x,t):= \rho_\dag \left( m(t)^{-\f1d} x,\, \ln m(t) \right),\\
&
\tilde{\tilde{p}}(x,t):=  \f{m'(t)}{m(t)}\cdot m(t)^{\f2d}
p_\dag \left( m(t)^{-\f1d} x,\, \ln m(t) \right)
\end{align*}
satisfy \eqref{eqn: a third equivalent dynamics}.
\end{proof}
\end{prop}

Using Proposition \ref{prop: master dynamics 2}, one can readily characterize the long-time behavior of the $\rho$-patch when $n_0>0$ is constant.
Indeed, when $n_0\geq 1$, the long-time dynamics of $\rho$ in \eqref{eqn: eta formula} and \eqref{eqn: original eqn} corresponds to infinite-time asymptotics of $\rho_\dag$ in \eqref{eqn: drift under potential}, which has been well-studied the literature, see e.g.\;\cite[Theorem 5.6]{AKY14}.
While for $n_0\in (0,1)$, we knew from Theorem \ref{thm: convergence when n_0 is less than 1} that $\rho$ converges to a compactly supported $\rho_\infty$ as $t\to +\infty$.
On the other hand, under the rescaling of Proposition \ref{prop: master dynamics 2}, such dynamics of $\rho$ actually corresponds to an excerpt of the master $\rho_\dag$-dynamics up to a finite time.
Therefore, we may characterize the long-time behavior of $\rho$ in the model \eqref{eqn: eta formula} and \eqref{eqn: original eqn} with any value of $n_0$ in the following unified way.

\begin{prop}
\label{prop: long-time asymptotics for constant n_0 no less than 1}
Suppose $n_0>0$ is constant in $\BR^d$.
Let $m(t)$ be defined in \eqref{eqn: volume}, and denote $\beta(t):= m(t)^{\f1d}$.
Assume $\Omega_0$ to be a bounded open set, such that $B_{r_1}(0)\subset \Omega_0 \subset B_{r_2}(0)$ for some $r_1,r_2>0$.
Let $r_\infty>0$ be defined such that $|B_{r_\infty}(0)| = |\Omega_0|$.

Let $\rho(x,t)$ solve \eqref{eqn: eta formula} and \eqref{eqn: original eqn} with $\rho_0 = \chi_{\Omega_0}$.
Then there exists a constant $C>0$ only depending on $r_1$ and $r_2$, but not on $n_0$, such that
\beqo
\beta(t) W_2\Big( \rho\big(\beta(t) x,t\big), \chi_{B_{r_\infty}(0)}(x) \Big)\leq C
\eeqo
for all $t\geq 0$.
Here $W_2$ denotes the 2-Wasserstein distance.
\end{prop}
\begin{proof}

Let us recall that, by virtue of Lemma \ref{lem: general test function}, $\rho(\beta(t) x,t)$ has the same total mass as $\chi_{B_{r_\infty}(0)}(x)$.

First we assume $n_0\geq 1$.

By comparison with the radial barriers, $\rho(\beta(t) x,t)$ and $\chi_{B_{r_\infty}(0)}(x)$ are both supported in $B_{r_2}(0)$ for all time.
Let $T>0$ satisfy $\beta(T) = \f{2r_2}{r_1}$.
Then it is obvious that the above inequality holds on $[0,T]$, with $C>0$ only depending on $r_1$ and $r_2$.

Next we consider the case $t\geq T$.
Assume $\rho(x,t) = \chi_{\Omega_t}(x)$.
By Corollary \ref{star:shaped}(b) and comparison with the radial barriers (see Remark \ref{rmk: radial solution}), for all $t\geq T$, $\Omega_t$ contains $B_{2r_2}(0)$ and thus it has Lipschitz boundary.
By Proposition \ref{prop: master dynamics 2},
\[
\rho_\dag(x,\ln m(T)) = \rho\big(\beta(T)x,T\big) = \chi_{\beta(T)^{-1}\Omega_T}(x) .
\]
Note that the rescaled set $\beta(T)^{-1}\Omega_{T}$ has Lipschitz boundary.
We then solve \eqref{eqn: drift under potential} starting from $t = \ln m(T)$ with ``initial data"  $\rho_\dag(x,\ln m(T))$.
By  \cite[Theorem 5.6 and its proof]{AKY14}, for all $t' \geq 0$,
\[
W_2\Big( \rho_\dag \big(x,\ln m(T)+t'\big), \chi_{B_{r_\infty}(0)}(x) \Big) \leq e^{-\f{t'}{d}}
W_2\Big( \rho_\dag \big(x,\ln m(T)\big), \chi_{B_{r_\infty}(0)}(x) \Big).
\]
Then the desired result follows from suitable change of variables.

By Proposition \ref{prop: master dynamics 2},
\beq
\rho_\dag(x,\ln m(t)) = \rho(\beta(t)x,t)\mbox{ holds for all }n_0>0.
\label{eqn: implication of master dynamics}
\eeq
Hence, the above argument essentially proves that, without assuming $\Omega_0$ has Lipschitz boundary (c.f.~\cite{AKY14}), for all $\tau \geq 0$,
\beqo
e^{\f{\tau}{d} } W_2\big( \rho_\dag (x, \tau), \chi_{B_{r_\infty}(0)}(x) \big)\leq C.
\eeqo
Hence, using \eqref{eqn: implication of master dynamics} again, the case $n_0 \in (0,1)$ is proved immediately.
\end{proof}


\subsection{Uniform free boundary regularity}
\label{sec: uniform boundary regularity}
In this section, we are going to prove uniform free boundary regularity up to $t = +\infty$ under the assumption that $n_0$ is constant in $\BR^d$.

Thanks to the master dynamics, the case $n_0<1$ is trivial.
This is because, up to a re-scaling in time (see Proposition \ref{prop: master dynamics}), its $\rho$-evolution in a time range of the form $(t,+\infty)$ corresponds to the $\rho$-evolution in a finite time range with a different $n_0\geq 1$.
The latter has already been characterized in Corollary \ref{cor:1}(a).

Also by the master dynamics, it suffices to study regularity of $\partial \Omega_t$ of the solution corresponding to \emph{one} arbitrary $n_0\geq 1$.
When $n_0> 1$, one can apply comparison principle and radial barriers constructed in Remark \ref{rmk: radial solution} to show that $\Omega_t$ expands exponentially fast. It is thus reasonable to expect uniform regularity of the free boundary after a suitable re-scaling.
Indeed, the key lies in the uniform Lipschitz estimate for $\partial \{w(\cdot,t)>0\}$ in Corollary \ref{star:shaped}(b).

\medskip

Let us point out that, while the master dynamics in rescaled variable \eqref{eqn: drift under potential} corresponds to a Hele-Shaw flow, the presence of the drift prevents us from directly applying existing regularity results (e.g.\;\cite{CJK07}) to our problem.





\begin{theorem}
\label{thm: uniform regularity}
Fix $n_0>1$.
Let $\Omega_0$, $r_1$, $r_2$, and $\beta(t)$ be given as in Proposition \ref{prop: long-time asymptotics for constant n_0 no less than 1}.
Let $\rho_0 = \chi_{\Omega_0}$.
Then there is $\alpha\in (0,1)$ and $T>0$ depending on $r_1$, $r_2$, $d$, and $n_0$, such that the followings hold for all $t \geq T$.
\begin{enumerate}[(a)]
\item
The rescaled set $\tilde{\Omega}_t:= \beta(t)^{-1}\Omega_t$ has uniformly $C^{1,\alpha}$-boundary;

\item
The rescaled nutrient variable $\tilde{n}(x,t):= n(\beta(t)x,t)$ is uniformly bounded in $C^{\alpha}(\{|x|\geq 2\beta^{-1}(t)r_2^2/r_1\})$.
\end{enumerate}
\end{theorem}

\begin{proof}

From comparison with radial barriers (see Theorem \ref{thm:contraction} and Remark \ref{rmk: radial solution}), we find that 
\beq
B_{\beta(t)r_1}(0)\subset \Omega_t = \{\rho(\cdot,t)=1\}
\subset B_{\beta(t)r_2}(0)
\label{eqn: lower and upper bound for the tumor domain}
\eeq
up to measure-zero set.
%
%
Also, for some $C>0$ depending on $n_0$ and $r_1$,
$$
p(x,t) \geq C \beta^2(t) \hbox{ if } |x| \leq \frac{r_1 \beta(t)}{2}.
$$
Hence, the re-scaled pressure variable $\tilde{p}(x,t):=\beta^{-2}(t) p(\beta(t)x, t)$ satisfies
$$
\tilde{p}(x,t) \geq C \hbox{ if } |x|\leq\frac{r_1}{2}.
$$
On the other hand, let $T(\cdot)$ be introduced in Corollary~\ref{star:shaped}.
From Corollary \ref{star:shaped}, for $t\geq T(2r_2)$, we know that $\partial\{p(\cdot,t)>0\}$ is a Lipschitz graph with respect to the radial direction, with the Lipschitz constant less than $O((\beta(t)r_1)^{-1} r_2 )$.
Therefore, 
$\tilde{\Omega}_t =\{\tilde{p}(\cdot,t) >0\} \subset B_{r_2}(0)$,
with $\partial\tilde{\Omega}_t$ being uniformly Lipschitz with respect to the radial direction.
When $t$ is suitably large, depending on $r_1$, $r_2$, and $n_0$, we have the Lipschitz constant to be small enough for applying a similar argument as in Proposition \ref{pressure:nondeg}.

Since $\tilde{p}$ is superharmonic in its positive set, arguing with Dahlberg's Lemma as in the proof of Proposition~\ref{pressure:nondeg}, we conclude that there is a constant $\alpha\in (0,1)$ that is independent of the time such that, for given $\tilde{x}_* \in \partial\tilde{\Omega}_t$ and for $t\geq  T(2r_2)$ we have
\[
\tilde{p}(x,t) \geq C d(x, \tilde{\Omega}_t^c)^{2-\alpha} \hbox{ in } B_1(\tilde{x}_*).
\]
In the original coordinate, this corresponds to
\[
p(x,t) \geq
C \beta(t)^{\alpha} d(x,\Omega_t^c)^{2-\alpha}
\]
for any $x\in B_1(\beta(t)\tilde{x}_*)$ and $t \geq T(2r_2)$.
Thus the barrier argument as in the proof of Proposition \ref{pressure:nondeg}, with $\bar{n}$ replaced by $\beta(t)^{\alpha}$, yields that, for any $x_*\in \partial\Omega_t$ with $t\geq T(2r_2)$,
\beq
B_r(x_*) \subset \Omega_{t + C(\frac{r}{\beta(t)})^{\alpha}}.
\label{eqn: nondegeneracy of speed}
\eeq
for sufficiently small $r$.  This further implies
\beq
|T_x - T_y| \leq C\left(\frac{|x-y|}{\min\{\beta(T_x), \beta(T_y)\}}\right)^{\alpha} \hbox{ for any } x,y\mbox{ with }T_x,T_y\geq T(2r_2).
\label{eqn: difference of arrival time}
\eeq

To justify this, we assume $T_x<T_y$ without loss of generality.
We first consider the case where $|x-y|$ is large.
By \eqref{eqn: lower and upper bound for the tumor domain},
if $T_*>0$ satisfies that
\[
\beta(T_x+T_*)r_1 \geq |y|,
\]
then $|T_x-T_y|\leq T_*$.
Also by \eqref{eqn: lower and upper bound for the tumor domain},
\[
|y|\leq |x-y|+|x|\leq |x-y|+\beta(T_x)r_2.
\]
Hence, we let $T_*$ satisfy
\[
\beta(T_x+T_*)r_1 = |x-y|+\beta(T_x)r_2,
\]
which implies (c.f.\;\eqref{eqn: volume})
\[
e^{(n_0-1)T_*}\leq \f{m(T_x+T_*)}{m(T_x)} = \left(\f{|x-y|}{\beta(T_x)r_1}+\f{r_2}{r_1}\right)^{d},
\]
Therefore,
\[
|T_x-T_y|\leq \f{d}{n_0-1}\ln \left(\f{|x-y|}{\beta(T_x)r_1}+\f{r_2}{r_1}\right),
\]
which implies \eqref{eqn: difference of arrival time} whenever $|x-y|/\beta(T_x)$ is sufficiently large.
Otherwise, if $|x-y|/\beta(T_x)\leq C$ where $C$ depends on $r_1$, $r_2$, $d$, $n_0$, and $\alpha$, we may apply \eqref{eqn: nondegeneracy of speed} to obtain \eqref{eqn: difference of arrival time}.

Under the assumption $T_x<T_y$, we have that
\[
\big|\eta(x,t)-\eta(y,t)\big| \leq n_0 \left(1-e^{-|T_x-T_y|} \right) e^{-(t-T_y)_+}\leq n_0\min\big\{1,|T_x-T_y|\big\}e^{-(t-T_y)_+}.
\]
Hence, thanks to \eqref{eqn: difference of arrival time}, for any $x,y$ such that $T_x,T_y\geq T(2r_2)$,
\[
\f{|\eta(x,t)-\eta(y,t)|}{|x-y|^\alpha}
\leq C \min\big\{|x-y|^{-\alpha},\beta(T_x)^{-\alpha}\big\} e^{-(t-T_y)_+}.
\]
Since we defined $\tilde{\eta}(x,t):= \eta(\beta(t)x,t)$,
\beq
\f{|\tilde{\eta}(\beta(t)^{-1}x,t)-\eta(\beta(t)^{-1}y,t)|}{\beta(t)^{-\alpha}|x-y|^\alpha}
\leq C \beta(t)^{\alpha}\max\big\{|x-y|,\beta(T_x)\big\}^{-\alpha} e^{-(t-T_y)_+}.
\label{eqn: C alpha norm of rescaled eta}
\eeq
We claim that
\[
\beta(T_y)\leq C\big(|x-y|+\beta(T_x)\big)\leq C \max\big\{|x-y|,\beta(T_x)\big\},
\]
where $C$ may depend on $r_1$, $r_2$, $d$, and $n_0$.
Indeed, by the estimate for $T_*$ derived above
\[
\f{\beta(T_y)}{\beta(T_x)} \leq C e^{\f{(n_0-1)}{d}|T_y-T_x|}
\leq C e^{\f{(n_0-1)}{d}T_*}
\leq C\left(\f{|x-y|}{\beta(T_x)}+1\right).
\]
Hence, when $t\leq T_y$, the right-hand side of \eqref{eqn: C alpha norm of rescaled eta} is bounded by a universal constant that only depends on  $r_1$, $r_2$, $d$, $n_0$, and $\alpha$.
We may further assume $\alpha$ to be suitably small so that the right-hand side of \eqref{eqn: C alpha norm of rescaled eta} is uniformly bounded for $t\geq T_y$.
Now noticing that $|x|\geq 2r_2^2/r_1$ guarantees $T_x\geq T(2r_2)$ (c.f.\;\eqref{eqn: lower and upper bound for the tumor domain}), we can conclude (b), i.e., $\tilde{\eta}(x,t)$ has uniform-in-time H\"{o}lder regularity for $t\geq T(2r_2)$ for $|x|\geq 2\beta^{-1}(t)r_2^2/r_1$.

\medskip

Lastly, observe that $\tilde{w}(x,t):= w(\beta(t)x,t)$ solves the obstacle problem
$$
\Delta \tilde{w} = \tilde{f}\chi_{\{\tilde{w}(\cdot,t)>0\}}, \mbox{ where } \tilde{f} = 1- \tilde{\rho}_0-\tilde{\eta}.
$$
Then using the regularity of $\tilde{\eta}$, and Theorem 7.2 of \cite{Blank}, we can conclude (a).
\end{proof}

Before ending this section, let us briefly discuss the uniform boundary regularity issue in the case of non-constant $n_0$.

If $\|n_0\|_{L^\infty}\geq 1$, long-time asymptotics of the $\rho$-patches can be rather complicated, as it does not rule out that $n_0$ could be less than $1$ in some areas.
It is not even clear whether the total mass of the tumor would diverge.
Suitable conditions need to be imposed on $n_0$ in order to make the question of uniform regularity more meaningful.

For $\|n_0\|_{L^\infty}<1$, Theorem \ref{thm: convergence when n_0 is less than 1} states that $w(\cdot,t)$ monotone increases to converge to $w_{\infty}$, which features bounded support.
Thus, the pressure as well as $n$ vanishes in $\Omega_t$ as time tends to infinity, and the regularizing effect of the pressure variable vanishes over time.
On the other hand, under suitable assumptions, $\{w_{\infty}>0\}$ features smooth free boundary.
To see this, recall that $w_{\infty}$ solves the obstacle problem (c.f.~\eqref{eqn: eqn for stationary state})
$$
\Delta w_{\infty} = (1-\rho_0 - n_0)\chi_{\{w_{\infty}>0\}}.
$$
When $\rho_0 = \chi_{\Omega_0}$ with $\Omega_0$ having smooth (say $C^{1,1}$) boundary, the set $\{w_{\infty}>0\}$ lies strictly outside of the support of $\rho_0$, and thus near its free boundary $w_{\infty}$ solves $\Delta w_{\infty} = (1-n_0)\chi_{\{w_{\infty}>0\}}$. It follows that, in the setting of Corollary~\ref{cor:1}(b), $\partial\{w_{\infty}>0\}$ is $C^{\infty}$ provided that $n_0$ is smooth.
Nevertheless, it remains open whether one can use this asymptotic regularity of the free boundary to show uniform regularity of $\partial\{w(\cdot,t)>0\}$ in time.

\appendix
\section{The Proof of Lemma~\ref{lem: comparison}}
\label{appendix: comparison proof}
The proof closely follows the argument in \cite[Sections 3 and 5]{PQV} (also see \cite[Proposition 5.1]{GKM}),  with some extra efforts for handling unboundedness of the spatial domain.
\begin{proof}

By definition, for any non-negative $\psi \in H^1(Q_T)$ such that $\psi(\cdot,T)=0$,
\beqo
\begin{split}
&\; \int_0^T\int_{\RR^d}\nabla \psi\cdot \nabla \big(p^0-p^1\big)- \big(\rho^0-\rho^1\big)\partial_t \psi \,dx \,dt \\
= &\;
\int_{\RR^d} \psi(x,0)\big(\rho_0^0(x)-\rho_1^0(x)\big)\,dx + \int_0^T\int_{\RR^d} \psi \big(f^0-f^1\big) \,dx\,dt
\leq 0.
\end{split}
\eeqo
Hence, for any $R>0$ and any non-negative $\psi \in H^1(\BR^d\times [0,T])$ supported in $B_R\times [0,T]$ such that $\psi(\cdot,T)=0$,
\beq
\int_0^T\int_{B_R}  (\rho^0-\rho^1)\partial_t \psi
+ (p^0-p^1)\Delta \psi \,dx\,dt
\geq
\int_0^T\int_{\partial B_R} \frac{\partial \psi}{\partial \nu} \cdot \big(p^0-p^1\big)\,d\sigma(x)\,dt.
\label{eqn: weak formulation inequality reformulated model problem}
\eeq

%
Define
\[
A = \f{\rho^0-\rho^1}{\rho^0-\rho^1+p^0-p^1},\quad B = \f{p^0-p^1}{\rho^0-\rho^1+p^0-p^1}.
\]
We define $A = 0$ whenever $\rho^0 = \rho^1$ (even when $p^0 =p^1$), and $B = 0$ whenever $p^0 = p^1$ (even when $\rho^0 = \rho^1$).
Since $p^0\in P_\infty (\rho^0)$ and $p^1\in P_\infty(\rho^0)$, we have $A,B\in [0,1]$.
Then \eqref{eqn: weak formulation inequality reformulated model problem} can be written as
\beq
\int_0^T\int_{B_R}  \big(\rho^0-\rho^1+p^0-p^1\big)\big(A\partial_t \psi
+ B\Delta \psi\big)\,dx\,dt
\geq
\int_0^T\int_{\partial B_R} \frac{\partial \psi}{\partial \nu} \cdot \big(p^0-p^1\big)\,d\sigma(x)\,dt.
\label{eqn: weak formulation inequality reformulated model problem}
\eeq

Let $G$ be a compactly supported non-negative smooth function in $Q_T$.
Assume it is supported in $B_{R_0}\times [0,T]$ for some $R_0>0$.
Take an arbitrary $R\geq 2R_0$.
As in \cite{PQV}, we introduce smooth positive approximations of $A$ and $B$, denoted by $A_{n,R}$ and $B_{n,R}$, such that for some universal $C_R>0$ that depends on $R$,
\[
A_{n,R}, B_{n,R}\in \left[\f1n, 1\right],\quad
\|A_{n,R}-A\|_{L^2(B_R\times [0,T])}\leq \f{C_R}n,\quad
\|B_{n,R}-B\|_{L^2(B_R\times [0,T])}\leq \f{C_R}n.
\]
In the view of \eqref{eqn: weak formulation inequality reformulated model problem}, let $\psi_{n,R}$ solve the (mollified) dual equation
\[
\begin{split}
&\partial_t \psi_{n,R} + \frac{B_{n,R}}{A_{n,R}}\Delta \psi_{n,R} = - G\mbox{ in }B_R\times [0,T],\\
&\psi_{n,R}|_{\partial B_R\times [0,T]} = 0,\quad \psi_{n,R}(\cdot,T) = 0\mbox{ in }B_R.
\end{split}
\]
$\psi_{n,R}$ is then a smooth function on $B_R\times [0,T]$.
Plugging it into \eqref{eqn: weak formulation inequality reformulated model problem} as the test function, we find that
\[
\int_0^T\int_{B_R}  \big(\rho^0-\rho^1\big)(-G)\,dx\,dt + \mathcal{E}_{n,R}
\geq
\int_0^T\int_{\partial B_R} \frac{\partial \psi_{n,R}}{\partial \nu} \cdot \big(p^0-p^1\big)\,d\sigma(x)\,dt,
\]
where
\[
\mathcal{E}_{n,R}:= \int_0^T\int_{B_R}  \big(\rho^0-\rho^1+p^0-p^1\big)\left(B-\f{AB_{n,R}}{A_{n,R}}\right)\Delta \psi_{n,R} \,dx\,dt.
\]
We can argue as in \cite{PQV,GKM} to show that $\psi_{n,R}$ is non-negative and uniformly bounded on $B_R\times [0,T]$, whose bound only depends on $G$ and $T$, but not on $n$ or $R$.
We can also prove that $\mathcal{E}_{n,R}\to 0$ as $n\to +\infty$.
Moreover,
\beq
\frac{\partial \psi_{n,R}}{\partial \nu} \leq C R^{-(d-1)}\mbox{ on }\partial B_R\times [0,T],
\label{eqn: bound on normal derivative model problem}
\eeq
where $C$ only depends on $d$, $T$, and $G$, but not on $n$ or $R$.
To prove \eqref{eqn: bound on normal derivative model problem}, we recall that $R\geq 2R_0$ and $\psi_{n,R}\leq C_*$, where $C_* = C_*( G,T)$.
Let $\tilde{\psi}_R$ solve
\[
\Delta \tilde{\psi}_{R} = 0\mbox{ on }B_R\backslash \overline{B_{R_0}},\quad \tilde{\psi}_{R}\big|_{\partial B_{R_0}} = C_*,\quad \tilde{\psi}_{R}\big|_{\partial B_R} = 0.
\]
Then we find
\[
\begin{split}
&\partial_t \big(\tilde{\psi}_R- \psi_{n,R}\big) + \frac{B_{n,R}}{A_{n,R}}\Delta \big(\tilde{\psi}_R- \psi_{n,R}\big)= 0 \mbox{ in } (B_R\backslash \overline{B_{R_0}})\times [0,T],\\
&\big(\tilde{\psi}_R- \psi_{n,R}\big)\big|_{\partial (B_R\backslash \overline{B_{R_0}})\times [0,T]} \geq 0,\quad \big(\tilde{\psi}_R- \psi_{n,R}\big) =0\mbox{ in }B_R\backslash \overline{B_{R_0}}.
\end{split}
\]
By the maximum principle, $\psi_{n,R}\leq \tilde{\psi}_R$ on $(B_R\backslash \overline{B_{R_0}})\times [0,T]$, and thus \eqref{eqn: bound on normal derivative model problem} follows.
In fact, when $d = 1,2$, the bound can be improved.

Combining the above estimates yields that, whenever $R\geq 2R_0$,
\[
\begin{split}
\int_0^T\int_{B_R} \big(\rho^0-\rho^1\big)(-G)\,dx\,dt
\geq &\; - CR^{-(d-1)} \int_0^T \int_{\partial B_R} \big|p^0-p^1\big|\,d\sigma(x)\,dt\\
\geq &\; - CR^{-\frac{d-1}{2}} \left(\int_0^T \int_{\partial B_R} \big|p^0-p^1\big|^2\,d\sigma(x)\,dt\right)^{\frac12},
\end{split}
\]
where $C$ only depends on $d$, $T$, and $G$.
By Definition \ref{def: weak solution p equation}, $p^i\in L^2(Q_T)$.
Sending $R\to +\infty$, we obtain that
\[
\begin{split}
\int_0^T\int_{\BR^d} \big(\rho^0-\rho^1\big)(-G)\,dx\,dt
\geq &\; - C \liminf_{R\to +\infty} R^{-\frac{d-1}{2}} \left(\int_0^T \int_{\partial B_R} \big|p^0-p^1\big|^2\,d\sigma(x)\,dt\right)^{\frac12} = 0.
\end{split}
\]
Since $G$ is an arbitrary compactly supported non-negative smooth function in $Q_T$, we conclude that $\rho^0 \leq \rho^1$ almost everywhere.
\end{proof}

\bibliography{tumor_refs}

\newcommand{\etalchar}[1]{$^{#1}$}
\def\cprime{$'$}
\providecommand{\bysame}{\leavevmode\hbox to3em{\hrulefill}\thinspace}
\providecommand{\MR}{\relax\ifhmode\unskip\space\fi MR }
\providecommand{\MRhref}[2]{%
  \href{http://www.ams.org/mathscinet-getitem?mr=#1}{#2}
}
\providecommand{\href}[2]{#2}
\begin{thebibliography}{DPMSV16}

\bibitem[AKY14]{AKY14}
Damon Alexander, Inwon Kim, and Yao Yao, \emph{Quasi-static evolution and
  congested crowd transport}, Nonlinearity \textbf{27} (2014), no.~4, 823.

\bibitem[BCMP73]{baiocchi}
Claudio Baiocchi, Valeriano Comincioli, Enrico Magenes, and Gianni~Arrigo
  Pozzi, \emph{Free boundary problems in the theory of fluid flow through
  porous media: Existence and uniqueness theorems}, Annali di Matematica Pura
  ed Applicata \textbf{97} (1973), 1--82.

\bibitem[BJST{\etalchar{+}}94]{Ben94}
Eshel Ben-Jacob, Ofer Schochet, Adam Tenenbaum, Inon Cohen, Andras Czirok, and
  Tamas Vicsek, \emph{Generic modelling of cooperative growth patterns in
  bacterial colonies}, Nature \textbf{368} (1994), no.~6466, 46--49.

\bibitem[Bla01]{Blank}
Ivan Blank, \emph{Sharp results for the regularity and stability of the free
  boundary in the obstacle problem}, Indiana University Mathematics Journal
  (2001), 1077--1112.

\bibitem[Caf98]{Caff98}
Luis~A Caffarelli, \emph{The obstacle problem revisited}, Journal of Fourier
  Analysis and Applications \textbf{4} (1998), no.~4, 383--402.

\bibitem[CJK07]{CJK07}
Sunhi Choi, David Jerison, and Inwon Kim, \emph{Regularity for the one-phase
  {Hele-Shaw} problem from a {Lipschitz} initial surface}, American journal of
  mathematics \textbf{129} (2007), no.~2, 527--582.

\bibitem[CVW87]{CVW87}
Luis~A Caffarelli, Juan~Luis V{\'a}zquez, and Noem{\i}~Irene Wolanski,
  \emph{Lipschitz continuity of solutions and interfaces of the $n$-dimensional
  porous medium equation}, Indiana University mathematics journal \textbf{36}
  (1987), no.~2, 373--401.

\bibitem[DP21]{david2021free}
Noemi David and Beno{\^\i}t Perthame, \emph{Free boundary limit of a tumor
  growth model with nutrient}, Journal de Math{\'e}matiques Pures et
  Appliqu{\'e}es \textbf{155} (2021), 62--82.

\bibitem[DPMSV16]{bv_ot}
Guido De~Philippis, Alp{\'a}r~Rich{\'a}rd M{\'e}sz{\'a}ros, Filippo
  Santambrogio, and Bozhidar Velichkov, \emph{{BV} estimates in optimal
  transportation and applications}, Archive for Rational Mechanics and Analysis
  \textbf{219} (2016), no.~2, 829--860.

\bibitem[EJ81]{EJ}
Charles~M Elliott and Vladim{\i}r Janovsk{\`y}, \emph{A variational inequality
  approach to {Hele-Shaw} flow with a moving boundary}, Proceedings of the
  Royal Society of Edinburgh Section A: Mathematics \textbf{88} (1981),
  no.~1-2, 93--107.

\bibitem[FK14]{FK}
William~M Feldman and Inwon~C Kim, \emph{Dynamic stability of equilibrium
  capillary drops}, Archive for Rational Mechanics and Analysis \textbf{211}
  (2014), no.~3, 819--878.

\bibitem[GKM22]{GKM}
Nestor Guillen, Inwon Kim, and Antoine Mellet, \emph{A {Hele-Shaw} limit
  without monotonicity}, Archive for Rational Mechanics and Analysis (2022),
  1--40.

\bibitem[JKT21]{JKT21}
Matt Jacobs, Inwon Kim, and Jiajun Tong, \emph{Darcy's law with a source term},
  Archive for Rational Mechanics and Analysis \textbf{239} (2021), no.~3,
  1349--1393.

\bibitem[JL20]{bfm}
Matt Jacobs and Flavien L\'{e}ger, \emph{A fast approach to optimal transport:
  The back-and-forth method}, Numer. Math. \textbf{146} (2020), no.~3,
  513–544.

\bibitem[JL22]{jacobs_lee}
Matt Jacobs and Wonjun Lee, \emph{An efficient numerical scheme for tumor
  growth models}, 2022.

\bibitem[JLL21]{bfm_gf}
{Jacobs, Matt}, {Lee, Wonjun}, and {L\'eger, Flavien}, \emph{The back-and-forth
  method for {Wasserstein} gradient flows}, ESAIM: COCV \textbf{27} (2021), 28.

\bibitem[Kim06]{K06}
Inwon~C Kim, \emph{Regularity of the free boundary for the one phase
  {Hele-Shaw} problem}, Journal of Differential Equations \textbf{223} (2006),
  no.~1, 161--184.

\bibitem[Kit97]{kitsu97}
So~Kitsunezaki, \emph{Interface dynamics for bacterial colony formation},
  Journal of the Physical Society of Japan \textbf{66} (1997), no.~5,
  1544--1550.

\bibitem[KK20]{KK20}
Inwon Kim and Dohyun Kwon, \emph{On mean curvature flow with forcing},
  Communications in Partial Differential Equations \textbf{45} (2020), no.~5,
  414--455.

\bibitem[KKP21]{KKP}
Inwon Kim, Dohyun Kwon, and Norbert Po{\v{z}}{\'a}r, \emph{On volume-preserving
  crystalline mean curvature flow}, Mathematische Annalen (2021), 1--42.

\bibitem[KMM{\etalchar{+}}97]{KMMUS}
K~Kawasaki, A~Mochizuki, M~Matsushita, T~Umeda, and N~Shigesada, \emph{Modeling
  spatio-temporal patterns generated bybacillus subtilis}, Journal of
  theoretical biology \textbf{188} (1997), no.~2, 177--185.

\bibitem[Mim04]{M04}
Masayasu Mimura, \emph{Pattern formation in consumer-finite resource
  reaction-diffusion systems}, Publications of the Research Institute for
  Mathematical Sciences \textbf{40} (2004), no.~4, 1413--1431.

\bibitem[MRCS14]{MRS14}
Bertrand Maury, Aude Roudneff-Chupin, and Filippo Santambrogio,
  \emph{Congestion-driven dendritic growth}, Discrete \& Continuous Dynamical
  Systems \textbf{34} (2014), no.~4, 1575.

\bibitem[PQV14]{PQV}
Beno{\^\i}t Perthame, Fernando Quir{\'o}s, and Juan~Luis V{\'a}zquez, \emph{The
  {Hele-Shaw} asymptotics for mechanical models of tumor growth}, Archive for
  Rational Mechanics and Analysis \textbf{212} (2014), no.~1, 93--127.

\bibitem[PTV14]{PTV14}
Beno{\^\i}t Perthame, Min Tang, and Nicolas Vauchelet, \emph{Traveling wave
  solution of the {Hele-Shaw} model of tumor growth with nutrient},
  Mathematical Models and Methods in Applied Sciences \textbf{24} (2014),
  no.~13, 2601--2626.

\bibitem[WW79]{dahlberg}
Guido Weiss and Stephen Wainger, \emph{Harmonic analysis in {Euclidean} spaces,
  part 1}, vol.~1, American Mathematical Soc., 1979.

\end{thebibliography}
\bibliographystyle{amsalpha}

\end{document}